\documentclass[12pt,a4paper]{article}

\usepackage[latin1]{inputenc}
\usepackage[english]{babel}
\usepackage[usenames, dvipsnames]{xcolor}
\usepackage{amssymb}
\usepackage{amsfonts}
\usepackage{amsmath}
\usepackage{amsthm}
\usepackage{epsfig}
\usepackage{graphics}
\usepackage{dsfont}
\usepackage{esdiff}

\usepackage{verbatim}

\usepackage{latexsym}

\newtheorem{theorem}{Theorem}[section]
\newtheorem{lemma}[theorem]{Lemma}
\newtheorem{remark}[theorem]{Remark}
\newtheorem{example}[theorem]{Example}
\newtheorem{proposition}[theorem]{Proposition}

\newtheorem{corollary}[theorem]{Corollary}

\allowdisplaybreaks[4]

\newcommand{\NN}{\mathbb{N}}

\newcommand{\RR}{\mathbb{R}}

\newcommand{\LL}{\mathbb{L}}

\newcommand{\cB}{\mathcal{B}}

\newcommand{\cE}{\mathcal{E}}

\newcommand{\cL}{\mathcal{L}}

\newcommand{\bN}{\mathbb{N}}
\newcommand{\bR}{\mathbb{R}}

\newcommand{\bC}{\mathbb{C}}
\newcommand{\law}{\mathcal{L}}

\begin{document}
\title{On exponential functionals of L\'evy processes}
\author{Anita Behme\thanks{Technische Universit\"at M\"unchen, Institut f\"ur Mathematische Statistik, Boltzmannstra\ss e 3, D-85748 Garching bei M\"unchen,
Germany, email: a.behme@tum.de, tel.: +49/89/28917424, fax:+49/89/28917435}$\,$  and Alexander Lindner\thanks{Technische Universit\"at Braunschweig, Institut f\"ur
Mathematische Stochastik, Pockelsstr. 14, D-38106
Braunschweig, Germany, email: a.lindner@tu-bs.de,
tel.:+49/531/3917575, fax:+49/531/3917564}
}
\date{\today}
\maketitle

\begin{abstract}
Exponential functionals of L\'evy processes appear as stationary distributions of generalized Ornstein-Uhlenbeck (GOU) processes.
In this paper we obtain the infinitesimal generator of the GOU process and show that it is a Feller process.
Further we use these results to investigate properties of the mapping $\Phi$, which maps two independent L\'evy processes to their
corresponding exponential functional, where one of the processes is assumed to be fixed.
We show that in many cases this mapping is injective and give the inverse mapping in terms of (L\'evy) characteristics. Also,
continuity of $\Phi$ is treated and some results on its range are obtained. \end{abstract}


\section{Introduction}\label{S1}
\setcounter{equation}{0}

The {\it exponential functional} of a bivariate L\'evy process $(\xi,\eta)^T = ((\xi_t,\eta_t)^T)_{t\geq 0}$ is defined as
\begin{equation} \label{eq-integral}
V_\infty = \int_{(0,\infty)}e^{-\xi_{t-}}d\eta_t.
\end{equation}
Necessary and sufficient conditions for the convergence of integrals of the
form $\int_{(0,t]}e^{-\xi_{s-}}d\eta_s$ as $t\to \infty$ for a bivariate L\'evy process $(\xi,\eta)^T$
were given by Erickson and Maller~\cite[Thm. 2]{ericksonmaller05}.
Distributional properties of exponential functionals have been studied in
various articles throughout the years by e.g. Paulsen \cite{paulsen93}, Yor \cite{yor}, Bertoin et al.~\cite{bertoinlindnermaller08}, Kondo et al.~\cite{kondomaejimasato06}, Lindner
and Sato \cite{lindnersato09}, Behme \cite{behme2011} and Kuznetsov et al.~\cite{savovetal} to name just a few.

Denote by $\cL(X)$ the law of a random variable $X$. In this paper, for a given one-dimensional L\'evy process $\xi$,  we will consider mappings like
\begin{eqnarray*}
\Phi_\xi : D_\xi  & \to & \mbox{set of probability distributions on $\bR$} ,\\
\cL(\eta_1) & \mapsto & \cL \left( \int_0^\infty e^{-\xi_{s-}} \, d\eta_s \right)
\end{eqnarray*}
defined on $D_\xi := \{ \cL (\eta_1) : \eta$ L\'evy process, independent of $\xi,$ such that $\int_0^\infty e^{-\xi_{s-}} \,
d\eta_s $ converges a.s.$\}$
and we will examine injectivity and continuity of such mappings and gather information about their ranges.
 In the case that $\xi_t=at$ is deterministic, it is well known that $D_\xi=\mbox{ID}_{\log}(\RR)$
is the set of real-valued infinitely divisible distributions with finite $\mbox{log}^+$-moment and that
$\Phi_\xi$ is an algebraic isomorphism between $\mbox{ID}_{\log}(\RR)$ and $L(\RR)$, the set of real-valued
selfdecomposable distributions \cite[Prop. 3.6.10]{jurekmason}.

We start with a short example of a special case to illustrate the kind of results we obtain, as well as the occuring problems.

\begin{example} \label{example-start}
Suppose $(\xi_t)_{t\geq 0}$ is a compound Poisson process with
intensity rate $\lambda$ and jump heights measure $\tau$. Let $\eta$
be a L\'evy process independent of $\xi$ such that $\cL(\eta_1)\in D_\xi$. Define $T_i$ to be
the time of the $i$th jump of $\xi$ with $T_0 := 0$. Then
$$V_\infty=\int_{(0,\infty)} e^{-\xi_{s-}} \, d\eta_s  = \sum_{i=0}^\infty \int_{(T_i,T_{i+1}]}e^{-\xi_{T_i}}d\eta_t =
 \sum_{i=0}^\infty \left(\prod_{k=1}^i e^{-\Delta \xi_{T_k}} \right) (\eta_{T_{i+1}}-\eta_{T_i}).$$
Since $(e^{-\Delta \xi_{T_i}},
\eta_{T_{i+1}}-\eta_{T_i})_{i=0,1,2,\ldots}$ is an i.i.d. sequence,
as e.g. in \cite{behme-et-al} we obtain from this the distributional fixed point equation
$$V_\infty\overset{d}= X V_\infty' + H$$
where $(X,H)\overset{d}= (e^{-\Delta \xi_{T_i}},\eta_{T_{i+1}}-\eta_{T_i})$ for $i=1,2,\ldots$ and $V_\infty\overset{d}=V_\infty'$ where $V_\infty'$ is independent of $(X,H)$.
In terms of characteristic functions this yields $\phi_{V_\infty}(u)= \phi_{XV_\infty'}(u)\phi_H(u)$ and
adding the fact that the characteristic function $\phi_\eta$ of the
L\'evy process $(\eta_t)_{t\geq 0}$ and the corresponding exponentially subordinated process $(H_t)_{t\geq 0}=(\eta_{\tau(t)})_{t\geq 0}$ with $\tau\sim\mbox{Exp}(\lambda)$
fulfill the equation
$$\phi_H(u)=\frac{\lambda}{\lambda -\log (\phi_\eta(u))}$$
(see e.g. \cite[p.10]{steutelvanharn}) we have
\begin{equation} \label{eq-motivation}
\log(\phi_\eta(u)) \phi_{V_\infty}(u)= \lambda \left( \phi_{V_\infty}(u)- \phi_{XV_\infty'}(u)\right)=
\lambda \int_{\RR} \left(E\left[e^{iuV_\infty} \right]- E\left[e^{iue^{-y}V_\infty} \right] \right) \tau(dy).
\end{equation}
\end{example}

Now, in the setting of the example {\it if we knew} that the characteristic function of $V_\infty$ is non-zero on a dense subset of $\RR$
this gave us a formula for the characteristic exponent of $\eta$ and thus injectivity of the mapping $\Phi_\xi$.
But in general the quotient of two characteristic functions does not necessarily yield a unique
solution as has already been remarked in \cite{loeve}. Examples for non-uniqueness of such quotients are also given in \cite{levy61}.

To obtain formulas like \eqref{eq-motivation} for general L\'evy processes $(\xi, \eta)^T$ we will strongly make use of the fact
that GOU processes are Markov processes. So, in Section \ref{sec-markov} we first compute the infinitesimal generator
of the GOU process and show that it is actually a Feller process. In Section \ref{sec-relations} these results
will be used to obtain formulas of the form \eqref{eq-motivation} for general, independent L\'evy processes $\xi$ and $\eta$.
Hereby we obtain a general formula for
$\log(\phi_{\eta}(u)) \phi_{V_\infty}(u)$ in terms of the characteristic triplet of $\xi$ and $\cL(V_\infty)$ as
given in Theorem \ref{thm-generatorlimit} and Corollary \ref{cor-generatorlimit} and on the other hand in Theorem \ref{thm-generatorlimit-2} we express
$\log(\phi_{-\xi}(u)) \phi_{\log|V_\infty|}(u)$ in terms of the characteristic triplet of $\eta$ and $\cL(V_\infty)$.\\
Further, Section \ref{sec-inj} is devoted to the study of injectivity, which - in view of the results of Section \ref{sec-relations} -
now reduces to an examination of when either $\phi_{V_\infty}(u)$ or $\phi_{\log|V_\infty|}(u)$ are non-zero on a dense subset of $\RR$.
We give various examples of when the mapping $\Phi_\xi$ or its counterpart $\tilde{\Phi}_\eta$ (which maps $\cL(\xi_1)$ to
$\cL(V_\infty)$ for  $\eta$ fixed) are injective and argue why injectivity cannot be obtained if $\xi$ and $\eta$ are
allowed to exhibit a dependence structure.\\
Section \ref{sec-ranges} then uses the previous results to obtain information on the ranges of $\Phi_\xi$ and $\tilde{\Phi}_\eta$. In
particular, Theorem \ref{thm-nonormal} shows that centered Gaussian distributions can only be obtained in the setting of (standard) OU processes,
i.e. for $\xi$ being deterministic and $\eta$ being a Brownian motion.\\
Finally, in Section \ref{sec-cont} we give conditions for continuity (in a weak sense) of the mappings  $\Phi_\xi$ and $\tilde{\Phi}_\eta$ and
give an example of $\Phi_\xi$ being not continuous.

\section{Some background on GOU processes and Notations}
\setcounter{equation}{0}

By the L\'evy-Khintchine formula (e.g. \cite[Thm. 8.1]{sato}) the {\it characteristic exponent} of an $\RR^d$-valued L\'evy process $X=(X_t)_{t\geq 0}$ is given by
\begin{eqnarray*}
 \psi_X(u)&:=&\log \phi_X(u) := \log E\left[e^{i\langle u, X_1 \rangle } \right]\\
&=& i \langle \gamma_X, u\rangle - \frac{1}{2} \langle u, A_X u  \rangle + \int_{\RR^d} (e^{i \langle u, x \rangle} -1 -i \langle u, x \rangle \mathds{1}_{|x|\leq 1}) \nu_X(dx)
\end{eqnarray*}
where $(\gamma_X, A_X, \nu_X)$ is the {\it characteristic triplet} of $X$. In case that $X$ is real valued we will usually replace $A_X$ by $\sigma^2_X$.
To simplify notations, we set $\nu(\{0\})=0$ for any L\'evy measure $\nu$.
If the L\'evy measure $\nu_X$ satisfies the condition $\int_{|x|\leq 1} |x|\nu_X(dx)<\infty$ we may also use the L\'evy-Khintchine formula in the form
\begin{eqnarray*}
 \psi_X(u)&=& i \langle \gamma^0_X, u\rangle - \frac{1}{2} \langle u, A_X u  \rangle +
\int_{\RR^d} (e^{i \langle u, x \rangle} -1 ) \nu_X(dx)
\end{eqnarray*}
and call $\gamma_X^0$ the {\it drift} of $X$. We refer to \cite{sato} for any further information on L\'evy processes. We write $\Delta Y_t = Y_t - Y_{t-}$ for
any c\`adl\`ag process $Y$.

Given a bivariate L\'evy process $((\xi_t,\eta_t)^T)_{t\geq 0}$ and
a random variable $V_0$ on the same probability space,
\begin{equation} V_t=e^{-\xi_t} \left( \int_0^t e^{\xi_{s-}}d\eta_s +V_0  \right),\quad t\geq 0,
\label{GOUdef} \end{equation}
defines the {\it generalized Ornstein-Uhlenbeck (GOU) process driven by $(\xi,\eta)^T$ with starting random variable $V_0$}.
In the case that $\xi_t=at$ is deterministic, the process $V_t$ is
usually called {\it Ornstein-Uhlenbeck-type process}, while if $(\xi_t, \eta_t)=(at,B_t)$ for $B$ a Brownian motion, $V_t$ is known as
{\it Ornstein-Uhlenbeck (OU) process}.

The GOU process driven by $(\xi,\eta)^T$ is the unique solution of
the stochastic differential equation
\begin{equation}\label{SDEGOU}
dV_t=V_{t-}dU_t+dL_t,\quad t\geq 0, \end{equation} for the bivariate
L\'evy process $((U_t,L_t)^T)_{t\geq 0}$ given by
\begin{equation} \label{eq-def-UL}
\left( \begin{array}{c} U_t \\ L_t  \end{array} \right) = \left(
\begin{array}{l} -\xi_t + \sum_{0< s \leq t} \left(e^{- \Delta \xi_s}
-1 + \Delta \xi_{s}\right) + t \,\sigma_{\xi}^2/2 \\
 \eta_t + \sum_{0 < s \leq t} (e^{-\Delta \xi_s}-1)
\Delta \eta_s  - t\,\sigma_{\xi,\eta}  \end{array}\right),\quad
t\geq 0,
\end{equation}
where $\sigma^2_{\xi}$ and $\sigma_{\xi,\eta}$ denote the $(1,1)$
and $(1,2)$ elements of the Gaussian covariance matrix
$A_{(\xi,\eta)}$. Equation \eqref{eq-def-UL} defines a bijection
between all bivariate L\'evy processes $(\xi,\eta)^T$ and all
bivariate L\'evy process $(U,L)^T$ such that $\nu_U((-\infty,-1]) =
0$. The upper line of \eqref{eq-def-UL} is equivalent to $e^{-\xi_t}
= \cE(U)_t$, where $\cE(U)_t$ is the stochastic exponential of
$U$, which is defined as the unique c\`adl\`ag solution $S$ of $S_t = 1 +
\int_{(0,t]} S_{s-} dU_s$ (see e.g. \cite[Thm.~II.37]{protter}). Equation \eqref{SDEGOU} has a solution
for any bivariate L\'evy process $(U,L)^T$ and any starting random
variable $V_0$ independent of $(U,L)^T$, which in the case
$\nu_U(\{-1\}) = 0$ is given by
\begin{equation} \label{GOUdef2}
V_t = \cE(U)_t \left( \int_{(0,t]} \cE(U)_{s-}^{-1} \, d\eta_s +
V_0\right),
\end{equation}
where $\eta_t = L_t - \sum_{0<s\leq t} (1+\Delta U_s)^{-1} \Delta
U_s \Delta L_s  - t\sigma_{U,L}$, see \cite[Thm.
2.1]{behmelindnermaller11}. The processes of the form
\eqref{GOUdef2} hence constitute a slightly larger class of
stochastic processes than the GOU processes defined by
\eqref{GOUdef} since they allow $U$ also to have jumps smaller than
$-1$. Obviously, the GOU process defined in \eqref{GOUdef} as well
as the process defined in \eqref{GOUdef2} are time homogeneous
Markov processes \cite[Lem. 3.3]{behmelindnermaller11}.

In \cite{lindnermaller05} necessary and
sufficient conditions for the existence of causal, strictly stationary solutions of
the generalized Ornstein-Uhlenbeck process \eqref{GOUdef} are given. In particular it is shown (\cite[Thm. 2.1]{lindnermaller05})
that if $(V_t)_{t\geq 0}$ is strictly stationary and causal, then $\int_{(0,t]} e^{-\xi_{s-}}dL_s$ with $L$ as defined in \eqref{eq-def-UL}
converges a.s. to a finite random variable as $t\to \infty$ and the stationary law $\mu$ is given by $\mu=\cL(V_\infty)$
for $V_\infty = \int_{(0,\infty)}e^{-\xi_{s-}}dL_s.$ Observe that $L=\eta$ if $\xi$ and $\eta$ are independent.

The space of continuous functions $\RR^d\to \RR$ is denoted by $C(\RR^d)$. The subspaces of bounded functions, functions vanishing at infinity
and functions with compact support are written as $C_b(\RR^d)$, $C_0(\RR^d)$ and $C_c(\RR^d)$, resp. For $n\in \NN$ we write $C^n(\RR^d)$
for the space of functions which are
$n$-times continuously differentiable. Functions in $C^n_b(\RR^d)$ are $n$-times continuously differentiable and the first $n$ derivatives are bounded.
 $C^n_0(\RR^d)$ and $C^n_c(\RR^d)$ are defined likewise. For any bounded function $f$ we let $\|f\|=\|f\|_\infty$ denote its supremum norm. We write ``$\stackrel{d}{=}$'' to denote equality in distribution of random
 variables, ``$\stackrel{d}{\to}$'' to denote convergence in
 distribution of random variables,
 {\it i.i.d.} for ``independent and identically distributed'', and  $\log^+ (x) = \log (\max\{x,1\})$ for $x\in \bR$. Throughout, the
 characteristic function of a random variable $X$ is denoted by $\phi_X(u) = E e^{iuX}$, $u\in
 \mathbb{R}$, and the Fourier transform of a finite measure $\mu$ on
 $(\bR,\cB_1)$ by $\widehat{\mu}(u) = \int_{\bR} e^{iux} \,
 \mu(dx)$. Here, $\cB_1$ denotes the Borel-$\sigma$-algebra  in
 $\bR$.

\section{Feller property and the infinitesimal generator of the GOU process}\setcounter{equation}{0}
\label{sec-markov}

Let $(X_t)_{t\geq 0}$ be a time homogeneous
Markov process on $\RR^d$ with semigroup $T_t$, i.e.
$$T_tf(x)=\int_{\RR^d} f(y)\mu_t(x,dy)=E^x[f(X_t)]$$
where $\mu_t(x,dy)=P(X_t\in dy|X_0=x)$ are the transition probabilities of $X$ and $f\in C_0(\RR^d)$.
Then $X$ is a {\it Feller process} in $\RR^d$ if its semigroup fulfills the {\it Feller properties}
\begin{eqnarray*}
 &\mbox{(F1)}& \quad T_t C_0(\RR^d) \subset C_0(\RR^d)\\
&\mbox{(F2)}& \quad T_tf \to f \mbox{ as } t\to 0 \quad \forall f\in
C_0(\RR^d),
\end{eqnarray*}
where the convergence under (F2) is meant to hold in the Banach space $(C_0(\RR^d), \|\cdot\|_\infty)$.

The {\it infinitesimal generator} $A^X$ of a Feller process $X$ is defined by
$$A^Xf=\lim_{t\to 0} \frac{T_tf-f}{t}$$
for all functions $f$ in the domain of $A^X$, i.e. all $f$ in
$$D(A^X)=\left\{f\in C_0(\RR^d), \lim_{t\to 0} \frac{T_tf-f}{t}
\mbox{ exists in } \|\cdot \|_\infty  \right\}.$$ A subspace $D$ of
$D(A^X)$ is said to be a {\it core} for the generator $A^X$, if the
closure of the restriction of $A^X$ to $D$ is equal to $A^X$.

Every L\'evy process $X$ is a Feller process. If the L\'evy process $X$ is real-valued its generator $A^X$ is given by (e.g. \cite[Thm. 31.5]{sato})
\begin{equation} \label{generatorlevy}
 A^X f(x)= \frac{1}{2}\sigma_X^2 f''(x) + \gamma_X f'(x) +
\int_{\RR} (f(x+y)-f(x)- yf'(x)\mathds{1}_{|y|\leq 1} )\nu_X(dy)
\end{equation}
and it holds $C_0^2(\RR)\subset D(A^X)$.

The generator of the OU process is well known in the literature, unlike the
generator of the GOU process, which is presented in the next theorem. For
L\'evy processes with finite second moment this generator is also
given in \cite[Thm. 4.6.1]{kolokoltsov} and the formula for the
generator can also be found in
 \cite[Exercise V.7]{protter} (containing a typo). The fact that GOU processes are Feller processes and the
 determination of the cores seems to be new.\\
Note that the equation $dV_t^x = x + \int_{(0,t]} V_{s-}^x \, dU_s +
dL_t$ can be written as \begin{equation} \label{eq-Schilling} dV_t^x
= x + \int_{(0,t]} g(V_{s-}^x) \, d(U_s,L_s)^T \end{equation} with
$g(u) = (u,1) \in \bR^{1\times 2}$. Solutions of
\eqref{eq-Schilling} with {\it bounded} and locally Lipschitz $g$
are well known to constitute Feller process (e.g. \cite[Cor.
3.3]{schillingschnurr}), but the function $u \mapsto (u,1)$ is not
bounded so that this theory cannot be applied. Further, in
\cite[Rem. 3.4]{schillingschnurr}  an example is given when $g$ is
not bounded and the corresponding solution fails to be a Feller
process.

\begin{theorem} \label{thm-generator-GOU}
 Let $(Z_t)_{t\geq 0}=((U_t,L_t)^T)_{t\geq 0}$ be a bivariate L\'evy process with characteristic triplet $(\gamma_Z,
A_Z, \nu_Z)$ where $\gamma_Z=(\gamma_U,\gamma_L)^T$,
$A_Z=\begin{pmatrix} \sigma_U^2 & \sigma_{U,L} \\ \sigma_{U,L} &
\sigma_L^2 \end{pmatrix}$ and $\nu_Z((dz_1,dz_2)^T)$ such that
$\nu_Z((-1,dz_2)^T)=0$. Then the process $(V_t^x)_{t\geq 0}$ defined by
\begin{equation} \label{GOUSDE-matrix}
 V^x_t=x+ \int_{(0,t]}  V_{s-}^x\, dU_s +L_t = x+\int_{(0,t]} g(V_{s-}^x)dZ_s,
 \quad t\geq 0,
\end{equation}
for $g(u)=(u,1)$
is a Feller process whose generator $A^V$ has a domain containing
$$S(\RR):=\left\{f\in C_0^2(\RR) : \lim_{|x|\to\infty} \left(|x f'(x)|+ |x^2 f''(x)|\right)=0 \right\}.$$
In particular $C_c^\infty (\bR) \subset C_c^2(\RR)\subset D(A^V)$.
For any $f\in S(\RR)$ the generator can be written as
\begin{eqnarray}
 A^Vf(x)&=& f'(x) g(x) \gamma_Z + \frac{1}{2} f''(x)\left(g(x) A_Z g(x)^T\right) \label{generatorGOU-UL-matrix} \nonumber \\
&& + \int_{\RR^2} \left( f(x+g(x)z)-f(x)-f'(x)g(x)z\mathds{1}_{|z|\leq 1}\right) \nu_Z(dz)\nonumber\\
&=& f'(x)(x\gamma_U+\gamma_L) + \frac{1}{2} f''(x)(x^2\sigma_U^2 + 2x\sigma_{U,L}+\sigma_L^2) \label{generatorGOU-UL}\\
&& + \int_{\RR^2} (f(x+xz_1+z_2) - f(x) - f'(x)(xz_1
+z_2)\mathds{1}_{|z|\leq 1} ) \nu_{U,L}(dz_1,dz_2). \nonumber
\end{eqnarray}
The spaces $S(\RR)$, $C_c^2(\RR)$ and $C_c^\infty(\bR)$ are cores for $A^V$.
\end{theorem}

\begin{proof} $\,$\\
(i) Let us
first establish the Feller property. It is well known that $V_t^x$
is a time homogenous Markov process (e.g. \cite[Lem.
3.3]{behmelindnermaller11}). By \eqref{GOUdef2}, $V_t^x$ is given by
$V_t^x = \cE(U)_t \left( x + \int_{(0,t]} \cE(U)_{s-}^{-1} \,
d\eta_s\right)$. Since $\cE(U)_t \neq 0$  as a consequence of
$\nu_U(\{-1\}) = 0$, we have $\lim_{|x|\to \infty} |V_t^x| = \infty$
and hence $\lim_{|x|\to \infty} f(V_t^x) = 0$ for any $f\in
C_0(\RR)$. By Lebesgue's dominated convergence theorem, this implies
$T_t f(x) = E [f(V_t^x)] \to 0$ as $|x|\to\infty$. The fact that for
bounded and continuous $f$ the mapping $x\mapsto E[f(V_t^x)]$ is
continuous is obvious using dominated convergence.
Thus $T_t$ maps $C_0(\RR)$ into $C_0(\RR)$ and (F1) is shown. (F2)
follows from (F1) and \cite[Thm. 3.15]{liggett}, observing that
for each $x\in \bR$, $V^x$ satisfies $P(V_0^x=x)=1$ and
$(V_t^x)_{t\geq 0}$ is adapted to the smallest filtration satisfying
the usual hypotheses induced by $((U_t,L_t)^T)_{t\geq 0}$, which is
right continuous.

(ii) Before we prove \eqref{generatorGOU-UL}, we give a bound for
the integrand appearing in \eqref{generatorGOU-UL} which will be
used throughout. Let $f \in S(\bR)$ and set
\begin{eqnarray}
K_1(f) & := & \sup_{y\in \bR} \left\{ |f'(y)| (1+|y|) + |f''(y)|
(1+|y|)^2\right\} < \infty \quad \mbox{and} \label{eq-K1}\\
K_2 & := & \frac12 \, \sup_{y\in \bR} \; \; \sup_{\zeta \in \bR:
|\zeta|\leq (1+|y|)/2} \frac{ (1+|y|)^2}{(1+|y+\zeta|)^2}< \infty.
\nonumber
\end{eqnarray} We claim that
\begin{eqnarray} \label{eq-majorant1}
\lefteqn{\left| f(x + xz_1 + z_2) - f(x) - f'(x) (x z_1 + z_2)
\mathds{1}_{|z|\leq 1} \right|} \\
&\leq & K_1(f) K_2 |z|^2 \mathds{1}_{|z|\leq 1/2} + K_1(f) |z|
\mathds{1}_{1/2 < |z| \leq 1} + 2 \|f\| \mathds{1}_{|z| > 1/2} \; \;
\forall \; z = (z_1,z_2)^T \in \bR^2, \; x \in \bR. \nonumber
\end{eqnarray}
Indeed, this is obvious for $|z|> 1/2$ since $|xz_1+z_2| \leq
\sqrt{1+x^2} |z| \leq (1+|x|) |z|$. For $|z|\leq 1/2$, by Taylor's
theorem there is $\zeta\in \bR$ with $0\leq |\zeta| \leq |xz_1+z_2|
\leq (1+|x|) |z| \leq (1+|x|)/2$ such that
\begin{eqnarray*}
\lefteqn{\left| f(x+xz_1 + z_2) - f(x) - f'(x) (xz_1 + z_2)\right|}
\\
& = & 2^{-1} |f''(x+\zeta)|  (xz_1+z_2)^2 \\
& \leq & 2^{-1} \left|f''(x+\zeta) (1+|x+\zeta|)^2\right| \,
\frac{(1+|x|)^2}{(1+|x+\zeta|)^2} |z|^2 \\
& \leq & K_1(f) K_2 |z|^2,
\end{eqnarray*}
which shows \eqref{eq-majorant1} also for $|z|\leq 1/2$. In
particular, the right hand side of \eqref{generatorGOU-UL} is in
$C_0(\bR)$ for $f\in S(\bR)$ by Lebesgue's dominated
convergence theorem.

(iii) Let us show \eqref{generatorGOU-UL}. Let $f\in
S(\RR)$, then by It\^o's formula (e.g. \cite[Thm.
II.32]{protter}) we have
\begin{eqnarray*}
 \lefteqn{f(V^x_t)-f(V^x_0)}\\
&=& \int_{(0,t]} f'(V^x_{s-}) dV^x_s + \frac{1}{2}\int_{(0,t]}
f''(V^x_{s-}) d[V^x,V^x]^c_s + \sum_{0< s\leq
t}\left(f(V^x_s)-f(V^x_{s-})-f'(V^x_s)\Delta V^x_s\right)
\end{eqnarray*}
and hence
\begin{eqnarray}
 T_tf(x)-f(x)&=& E \left[f(V^x_t)-f(V^x_0)\right] \nonumber\\
&=& E \left[\int_{(0,t]} f'(V^x_{s-}) dV^x_s + \sum_{0< s\leq t}\left(f(V^x_s)-f(V^x_{s-})-f'(V^x_{s-})\Delta V^x_s\right)\right] \nonumber \\
&& + \frac{1}{2}E \left[\int_{(0,t]} f''(V^x_{s-}) d[V^x,V^x]^c_s \right]\nonumber\\
&=:& \mbox{I}_t + \mbox{II}_t, \quad \mbox{say.}
\label{eq-ito-numbers}
\end{eqnarray}
Observe that $dV_s^x = g(V_{s-}^x) d Z_s$ and $\Delta V^x_s=
g(V^x_{s-})\Delta Z_s$. Since $Z$ is a L\'evy process, by the
L\'evy-It\^o decomposition (e.g. \cite[Thm. 2.4.16]{applebaum}) we
can write $Z_t = \gamma_Z t + M_t + \sum_{0<s\leq t} \Delta Z_s
\mathds{1}_{|\Delta Z_s| > 1}$, where $(M_t)_{t\geq 0}$ is a square
integrable martingale with expectation 0. Hence we obtain for the
first term
\begin{eqnarray*}
 \mbox{I}_t
&=& E \left[\int_{(0,t]} f'(V^x_{s-})g(V^x_{s-}) \gamma_Z ds\right] + E \left[\int_{(0,t]} f'(V^x_{s-}) g(V^x_{s-}) dM_s\right]\\
&& + E \left[\sum_{0<s\leq t} f'(V^x_{s-}) g(V^x_{s-}) \Delta Z_s \mathds{1}_{|\Delta Z_s| > 1}+ \sum_{0< s\leq t}\left(f(V^x_s)-f(V^x_{s-})-f'(V^x_{s-})g(V^x_{s-})\Delta Z_s\right) \right].\\
\end{eqnarray*}
Since $M$ is a square integrable martingale with expectation 0 and
since $s\mapsto f'(V_{s-}^x) g(V_{s-}^x)$ is bounded because of
$f\in S(\bR)$, the process $t\mapsto \int_{(0,t]} f'(V^x_{s-})
g(V^x_{s-}) dM_s$ is a martingale with expectation 0 (e.g.
\cite[Prop. 2.24]{medvegyev}). Hence we conclude
\begin{eqnarray*}
 \mbox{I}_t &=& \int_{(0,t]} E \left[ f'(V^x_{s-})g(V^x_{s-})\right] \gamma_Z \, ds\\
&& + E \left[ \sum_{0< s\leq t}  \left( f(V^x_{s-} + g(V_{s-}^x)
\Delta Z_s)
-f(V^x_{s-})-f'(V^x_{s-})g(V^x_{s-})\Delta Z_s \mathds{1}_{|\Delta Z_s| \leq 1}\right) \right]\\
&=& \int_{(0,t]} E \left[ f'(V^x_{s-})g(V^x_{s-})\right] \gamma_Z
ds\\
 && + E \left[\int_{(0,t]} \int_{\RR^2} \left( f\left(V^x_{s-}+ g(V^x_{s-})z
\right)-f(V^x_{s-})-f'(V^x_{s-})g(V^x_{s-})z\mathds{1}_{| z|\leq 1}
\right)\nu_Z(dz) ds\right],
\end{eqnarray*}
where we used the compensation formula (e.g. \cite[Thm.
4.4]{kyprianou}), which may be applied since
$E\int_{(0,t]}\int_{\bR^2} \left|f\left(V^x_{s-}+
g(V^x_{s-})z\right)
-f(V^x_{s-})-f'(V^x_{s-})g(V^x_{s-})z\mathds{1}_{| z|\leq 1}
\right|\nu_Z(dz) ds$ is finite by  \eqref{eq-majorant1}. Using the
continuity of $s\mapsto V_s^x$ at $s=0$ and again the bound from
\eqref{eq-majorant1}, it follows from Lebesgue's dominated
convergence theorem that
$$\lim_{t\to 0} t^{-1} {\rm{I}}_t = f'(x) g(x) \gamma_Z + \int_{\bR^2} \left(
f(x+g(x) z) - f(x) - f'(x) g(x) z \mathds{1}_{|z|\leq 1} \right)\,
\nu_Z(dz).$$

For the second term in Equation
\eqref{eq-ito-numbers} observe that by \cite[Eq. (4)]{karandikar}
\begin{eqnarray*}
 [V^x,V^x]^c_s&=& \left[ x + \int_{(0,\cdot]} g(V^x_{u-})dZ_u ,x + \int_{(0,\cdot]} g(V^x_{u-})dZ_u\right]_s^c\\
&=& \int_{(0,s]} g(V^x_{u-}) \,d[Z, Z^T]_u^c \,g(V^x_{u-})^T ,
\end{eqnarray*}
and since $[Z,Z^T]^c_u = A_Z u$ it follows
$$\mbox{II}_t=\frac{1}{2} E \left[\int_{(0,t]} f''(V^x_{s-}) g(V^x_{u-})\, A_Z  \,g(V^x_{u-})^T \, du\right].$$
Together with the obtained formula for $\mbox{I}_t$, and inserting
the definition of $g$ and $Z$, this shows that $\lim_{t\to 0} t^{-1}
({\rm{I}}_t + {\rm{II}}_t)$ is given by the right hand side of
\eqref{generatorGOU-UL}. Since $V$ is a Feller process, and since
the right hand side of \eqref{generatorGOU-UL} is in $C_0(\bR)$ for
$f\in S(\bR)$ by \eqref{eq-majorant1}, this pointwise limit
is actually uniform in $x$ (e.g. \cite[Lem. 31.7]{sato}), so that
$S(\bR)$ is contained in the domain of the generator of $V$
and that $A^V f$ is given by \eqref{generatorGOU-UL} for all $f\in
S(\RR)$.

(iv) We now show that $S(\bR)$ is a core for $A^V$ under the extra
assumption  that $E U_1^2 < \infty$ and $E L_1^2 < \infty$. Denote
$A_t = \cE(U)_t$ and $B_t = \cE(U)_t \int_{(0,t]} \cE(U)_s^{-1} \,
d\eta_s$. Then $B_t \stackrel{d}{=} \int_{(0,t]} \cE(U)_{s-} \,
dL_s$ by \cite[Lem. 3.1]{behmelindnermaller11}. Then $E A_t^2 <
\infty$ and $E B_t^2 < \infty$ as a consequence of Proposition 3.1
and Lemma 6.1 in \cite{behme2011} together with \cite[Thm.
25.18]{sato}. We conclude that $\diffp[2]{}{x} T_t f(x)$ exists for
$f\in S(\bR)$ and that
\begin{eqnarray*}
\diffp{}{x} T_tf(x) & = & \diffp{}{x} E[f(A_t x+B_t)] = E[A_t f'(A_t
x+B_t)] \quad \mbox{and}\\
\diffp[2]{}{x} T_t f(x) & = & \diffp[2]{}{x} E [f(A_tx + B_t)] = E
[A_t^2 f''(A_tx + B_t)].
\end{eqnarray*}
Since $E A_t^2 < \infty$, the mapping $x\mapsto \diffp[2]{}{x} T_t f(x)$ is
obviously continuous, so that $T_t S(\bR) \subset C^2(\bR)
\cap C_0(\bR)$. Using that $E |B_t| < \infty$ and $\lim_{|y|\to
\infty} |y f'(y)| = 0$ for $f\in S(\bR)$, we further see by
dominated convergence that
\begin{eqnarray*}
\left| x \diffp{}{x} T_t f(x)\right|
&\leq& E\left[\left|A_t x f'(A_t x + B_t)\right|\right] \\
&\leq& E\left[\left| (A_t x + B_t) f'(A_t x + B_t)\right|\right] + E\left[\left|B_t f'(A_t x + B_t)\right|\right]\\
&\to& 0,\quad \mbox{as }|x|\to\infty .
\end{eqnarray*}
In the same way one can check that $\left|x^2 \diffp[2]{}{x}
T_tf(x)\right|\to 0$ as $|x|\to \infty$ such that $T_tS(\RR)
\subset S(\RR)$. By \cite[Prop. 1.3.3]{ethierkurtz}  we thus
obtain that $S(\RR)$ is a core for $A^V$, provided that $E
U_t^2 < \infty$ and $E L_t^2 < \infty$.

(v) Now we drop the assumption that $E U_1^2 + E L_1^2 < \infty$ and
show that $S(\bR)$ is a core for $A^V$. Similarly to the
proof of Theorem 3.1 in Sato and Yamazato \cite{satoyamazato}, for
$f \in S(\bR)$ denote the right hand side of
\eqref{generatorGOU-UL} by $G f(x)$ and define
\begin{eqnarray*} G_0 f(x) & := &  f'(x)(x\gamma_U+\gamma_L) + \frac{1}{2} f''(x)(x^2\sigma_U^2 + 2x\sigma_{U,L}+\sigma_L^2) \\
&& + \int_{\{ z \in \RR^2: |z|\leq 1\}} \left(f(x+xz_1+z_2) - f(x) -
f'(x)(xz_1 +z_2)\mathds{1}_{|z|\leq 1} \right) \nu_{U,L}(dz_1,dz_2).
\end{eqnarray*}
For $f\in C_0(\bR)$, denote further
$$W f(x) := \int_{\{ z \in \bR^2 : |z|> 1\}} \left( f(x+ x z_1 + z_2) -
f(x) \right) \, \nu_{U,L}(dz_1,dz_2).$$ Then $W: C_0(\bR) \to
C_0(\bR)$ is a bounded linear operator, and from (iii) we know that
$A^V f = G f = G_0 f + W f$ for $f\in S(\bR)$. Consider the process
$V_{(0)}$ defined by ${V}_{(0),t}^x = x + \int_{(0,t]}
{V}_{(0),s-}^x \, d\tilde{U_s} + \tilde{L}_t$, where
$(\tilde{U},\tilde{L})^T$ is a L\'evy process with characteristic
triplet $(\gamma_Z, A_Z, \mathds{1}_{|z|\leq 1} \nu_Z(dz))$. Again
by (iii), $G_0 f = A^{V_{(0)}} f$ for $f\in S(\bR)$, and from (iv)
we know that $S(\bR)$ is a core for $A^{V_{(0)}}$, so that the
closure $\overline{G_0}$ of ${G_0}$ is $A^{{V_{(0)}}}$, in
particular $D (\overline{{G_0}}) = D(A^{{V_{(0)}}})$. Since $G f =
{G_0} f + W f$ for $f\in S(\bR)$ and $W$ is bounded, it follows that
the closure $\overline{G}$ of $G$ satisfies $\overline{G} =
\overline{{G_0}} + W$, in particular $D(A^{{V_{(0)}}}) = D
(\overline{{G_0}}) = D(\overline{G})$. Since $A^V$ is a closed
operator, we further know that $D(\overline{G}) \subset D(A^V)$ and
that $A^V$ is a closed extension of $\overline{G}$. From the
Hille-Yosida theorem (e.g. \cite[Thm. 1.2.6]{ethierkurtz}) it
follows that for every $\lambda
> 0$, $\lambda \,\mbox{Id} - \overline{{G_0}}: D(A^{{V_{(0)}}}) =
D(\overline{{G_0}}) \to C_0(\bR)$, $f\mapsto \lambda f - \overline{{G_0}} f$
is a bijection with bounded inverse (the resolvent) satisfying $\| (\lambda\,
\mbox{Id} - \overline{{G_0}})^{-1}\| \leq \lambda^{-1}$. For $\lambda_0>
\|W\|$, it then follows from a perturbation result for closed linear
operators (e.g. \cite[Thm. IV.1.16]{kato}), that also $\lambda_0\,
\mbox{Id} - \overline{G} = \lambda_0\,  \mbox{Id} - \overline{{G_0}} - W:
D(\overline{G}) = D(\overline{{G_0}}) \to C_0(\bR)$ is a bijection with
bounded inverse. Since $A^V$ is a closed extension of $\overline{G}$ and
also $\lambda_0 \, \mbox{Id} - A^V: D(A^V) \to C_0(\bR)$ is a
bijection (e.g. \cite[Prop. 1.2.1]{ethierkurtz}), we must have
$D(\overline{G}) = D(A^V)$ and hence $\overline{G} = A^V$. This shows that
$S(\bR)$ is a core for $A^V$.

(vi) Finally, we show that $C_c^2(\bR)$ and $C_c^\infty(\bR)$ are
cores for $A^V$. Let $h$ be a function in $C^\infty_c$ with $h(x)=1$
if $|x|\leq 1$ and $h(x)=0$ if $|x|\geq 2$. Define $h_n(x)=h(x/n)$
and for any $f\in S(\RR)$ set $f_n(x)=f(x)h_n(x)$. Then
$f_n\in C_c^2(\RR)$ and we obtain that $f_n\to f$, $f_n'\to f'$,
$f_n''\to f''$, $xf_n'(x)\to xf'(x)$, $xf_n''(x)\to xf''(x)$ and
$x^2f_n''(x)\to x^2f''(x)$ uniformly in $x$ as $n\to \infty$. In
particular, with $K_1(\cdot)$ as defined in \eqref{eq-K1}, we see
that $K_1(f_n)$ is bounded in $n$ and hence we conclude with
\eqref{eq-majorant1} that $A^V f_n\to A^V f$ uniformly as $n\to
\infty$. This shows that  $C_c^2(\RR)$ is a core for $A^V$. Finally,
for $f\in C_c^2(\bR)$ there are functions $g_n\in C_c^\infty(\bR)$
with uniformly bounded supports such that $g_n\to f$, $g_n'\to f'$
and $g_n'' \to f''$ uniformly as $n\to\infty$, hence also $x g_n'(x)
\to x f'(x)$, $x g_n''(x) \to x f''(x)$ and $x^2 g_n''(x) \to x^2
f''(x)$ uniformly in $x$ as $n\to \infty$. Again, this gives $A^V
g_n \to A^V f$ uniformly as $n\to\infty$ so that $C_c^\infty$ is a
core for $A^V$.
\end{proof}
The following corollary is immediate from Theorem
\ref{thm-generator-GOU}.

\begin{corollary}
In the setting of Theorem \ref{thm-generator-GOU}, if $U$ and $L$ are
additionally independent, Equation \eqref{generatorGOU-UL} simplifies to
\begin{eqnarray}
 A^Vf(x)&=& f'(x)(x\gamma_U+\gamma_L) + \frac{1}{2} f''(x)(x^2\sigma_U^2 +\sigma_L^2)\nonumber \\
&& + \int_{\RR} (f(x+xy) - f(x) - f'(x)xy \mathds{1}_{|y|\leq 1}) \nu_U(dy) \nonumber\\
&& + \int_\RR (f(x+y) - f(x) - f'(x)y\mathds{1}_{|y|\leq 1}) \nu_{L}(dy)\nonumber \\
&=& A^Lf(x) + f'(x) x\gamma_U + \frac{1}{2} f''(x) x^2\sigma_U^2  \label{generatorGOU-UL-ind}\\
&& + \int_{\RR} (f(x+xy) - f(x) - f'(x)xy \mathds{1}_{|y|\leq 1})
\nu_U(dy) .\nonumber
\end{eqnarray}
\end{corollary}\bigskip

\begin{corollary}\label{coro-generatorGOUxieta}
Let $(\xi_t)_{t\geq 0}$ and $(\eta_t)_{t\geq 0}$ be two independent
L\'evy processes and let $(V^x_t)_{t\geq 0}$ be the generalized
Ornstein-Uhlenbeck process driven by $(\xi, \eta)^T$ with starting
point $x$ as defined in \eqref{GOUdef}. Then $(V^x_t)_{t\geq 0}$ is
a Feller process whose generator has a domain containing $S(\RR)$,
and $S(\RR)$, $C_c^2(\bR)$ and $C_c^\infty(\bR)$ are cores for
$A^V$. For any $f\in S(\RR)$ the generator can be written as
\begin{eqnarray}
 A^Vf(x)&=& A^{\eta}f(x) - f'(x) x\gamma_\xi + \frac{1}{2} (f''(x) x^2 + f'(x)x) \sigma_\xi^2  \label{generatorGOU-xieta-ind}\\
&& + \int_{\RR} (f(xe^{-y}) - f(x) + f'(x)xy \mathds{1}_{|y|\leq 1})
\nu_\xi (dy).\nonumber
\end{eqnarray}
If $f\in S(\bR)$ and $f(0) = 0$, define
$\tilde{f}(x)=f(e^x)$ and $\tilde{\tilde{f}}(x)=f(-e^x)$. Then
$\tilde{f}, \tilde{\tilde{f}} \in C_0^2(\bR) \subset D (A^{-\xi})$,
and for such $f$ Equation \eqref{generatorGOU-xieta-ind} can be rewritten
as
\begin{equation} \label{generatorGOU-xieta-ind-short}
A^V f(x) =  A^{\eta}f(x)+ A^{-\xi} \tilde{f}(\log x)\mathds{1}_{x>0}
+ A^{-\xi} \tilde{\tilde{f}}(\log |x|)\mathds{1}_{x<0}.
\end{equation}
\end{corollary}

\begin{proof}
Since $(V^x_t)_{t\geq 0}$ fulfills \eqref{GOUSDE-matrix} for
$(U,L)^T$ defined in \eqref{eq-def-UL} the Feller property as well
as the statements on the domain and cores of the generator follow
directly from Theorem \ref{thm-generator-GOU}. Also observe from
\eqref{eq-def-UL} that in the independent case we have $\eta_t=L_t$
and thus $A^{\eta}=A^L$ whereas the relation between $\xi$ and $U$
yields $\nu_U((-\infty, -1])=0$. In \cite[Lem.
3.4]{behmelindnermaller11} we have computed the characteristic
triplet of $\xi$ in terms of the characteristic triplet of $U$ (the
$\hat{U}$ used there is equal to $\xi$ whenever
$\nu_U((-\infty,-1])=0$). Using these relations
one obtains \eqref{generatorGOU-xieta-ind} from \eqref{generatorGOU-UL-ind} by standard computations. \\
Finally the fact that $\tilde{f}, \tilde{\tilde{f}} \in C_0^2(\bR)$
if $f\in S(\bR)$ such that $f(0) = 0$ and the validity of
\eqref{generatorGOU-xieta-ind-short} may be checked directly from
\eqref{generatorlevy} using the definitions of $\tilde{f}$ and
$\tilde{\tilde{f}}$.
\end{proof}

\begin{remark}
In \cite{savovetal} the exponential functional for independent processes $\xi$ and $\eta$ is studied.
Under the condition of finite first moments of $\xi$ and $\eta$, the authors prove
that for suitable functions $f$ with support on the positive half line the generator of the GOU process can be written as
\begin{equation}
A^{V}f(x) = A^{-\xi}\tilde{f}(\log x) + A^{\eta} f(x)
\label{eq-savov-et-al}
\end{equation}
where $\tilde{f}(x)=f(e^x)$ and $A^{-\xi}$ and $A^{\eta}$ are the
generators of $-\xi$ and $\eta$ respectively. Remark that the $\xi$
used by the authors corresponds to $-\xi$ in our notation. The
formula \eqref{eq-savov-et-al} for positive $x$ is also obtained in
\cite[Proof of Thm. 1]{carmona-unpublished}.
\end{remark}

\section{Relations between the exponential functional and the driving L\'evy processes}\setcounter{equation}{0}
\label{sec-relations}

It is basic knowledge in the theory of Markov processes (see e.g.
\cite[Prop. 4.9.2]{ethierkurtz}), that if $\mu$ is an
invariant measure for the Markov process $X$ with strongly
continuous contraction semigroup $T_t$ and generator $A$, i.e. if
$\mu(B)=\int \mu_t(x,B) \mu(dx)$ for all Borel sets $B$, then
\begin{equation} \label{generatorstat}
 \int_{\RR^d} Af(y)\mu(dy)= 0 \quad \forall f\in D(A).
\end{equation}
Conversely, if \eqref{generatorstat} holds, $\mu$ is an invariant measure under some additional conditions.
In the special case of Feller processes Equation \eqref{generatorstat} holds if and only if $\mu$ is an invariant
measure of the corresponding process $X$ \cite[Thm. 3.37]{liggett}.

In \cite{carmonapetityor} and \cite{carmonapetityorLP} the authors
make use of Equation \eqref{generatorstat} to obtain the density of
a specific stationary generalized Ornstein-Uhlenbeck process. More
precisely they obtain the density of the exponential functional in
the special case that $\xi$ is a Brownian motion with drift and $\eta$ is deterministic.

Let $(V_t)_{t\geq 0}$ be a GOU process as defined in \eqref{GOUdef}
or even a process as defined in \eqref{GOUdef2}, fulfilling the SDE
\eqref{SDEGOU}, with $\nu_U( \{-1\}) = 0$. Assume that $U$ and $L$
are independent, i.e. the generator of $(V_t)_{t\geq 0}$ is given by
\eqref{generatorGOU-UL-ind} for $f\in S(\bR)$. Let
$\mu=\cL(V_\infty)$ be the invariant measure of $(V_t)_{t\geq 0}$,
assuming its existence. Then by \eqref{generatorstat} we obtain for
any $f \in S(\bR) \subset D(A^V)$
\begin{eqnarray} \label{generatorinteq}
 0 &=& \int_{\RR} A^V f(x) \mu (dx) \nonumber\\
&=&  \int_{\RR}A^{L} f(x) \mu(dx)+ \gamma_U \int_{\RR} f'(x) x \,\mu (dx) + \frac{\sigma_U^2}{2} \int_{\RR} f''(x) x^2 \,\mu(dx) \label{eq-intgenerator} \\
&& + \int_{\RR} \int_{\RR\backslash\{-1\}} (f(x+xy) - f(x) - f'(x)xy \mathds{1}_{|y|\leq 1}) \nu_U(dy) \mu(dx).\nonumber
\end{eqnarray}

This and the previous  results allow to establish relationships
between the characteristic functions of $V_\infty$, $U$ and $L$, as done in the following. Recall that $\psi_X(u)=\log E[e^{iuX_1}]$ is the characteristic exponent of the L\'evy process $X$.

\begin{theorem}\label{thm-generatorlimit}
Let $(U_t)_{t\geq 0}$ and $(L_t)_{t\geq 0}$ be two independent
L\'evy processes with $\nu_U(\{-1\}) = 0$ and such that
$V_\infty=\int_0^\infty \cE(U)_{s-}dL_s$ converges to a finite
random variable. Then $\mu=\cL(V_\infty)$ is the
invariant law of the
process $(V_t)_{t\geq 0}$ defined by \eqref{GOUdef2}.\\
Let $h\in C^\infty_c(\RR)$ be such that $h(x)=1$ for $|x|\leq 1$ and
$h(x)=0$ for $|x|\geq 2$ and set $h_n(x):=h(\frac{x}{n})$ and
$f(x)=e^{iux}$, $f_n(x)=f(x)h_n(x)$ for $u\in \bR$. Then
\begin{eqnarray}
\psi_L(u)\phi_{V_{\infty}}(u)&=&\lim_{n\to \infty} \left( -\gamma_U \int_{\RR} xf_n'(x) \,\mu (dx) - \frac{\sigma_U^2}{2} \int_{\RR} x^2f_n''(x) \,\mu(dx) \right. \label{eq-relation1}\\
&& \left. - \int_{\RR} \int_{\RR} (f_n(x+xy) - f_n(x) - xyf_n'(x)
\mathds{1}_{|y|\leq 1}) \nu_U(dy) \mu(dx)\right). \nonumber
\end{eqnarray}
If additionally $E [V_\infty^2] = \int_{\bR} x^2 \, \mu(dx) <
\infty$, then
\begin{eqnarray}
\psi_L(u) \phi_{V_\infty} (u)
&=& -iu\gamma_U E\left[V_\infty e^{iuV_\infty} \right] + \frac{\sigma_U^2 u^2}{2} E\left[V_\infty^2 e^{iuV_\infty} \right] \label{eq-relation2}\\
&&- \int_{\RR} \left(\phi_{V_{\infty}}(u(1+y)) - \phi_{V_{\infty}}(u) - iuE\left[V_\infty e^{iuV_\infty} \right] y \mathds{1}_{|y|\leq 1}\right) \nu_U(dy) \nonumber \\
 &=& - u \gamma_U \phi_{V_\infty}'(u) - \frac{\sigma_U^2 u^2}{2} \phi_{V_\infty}''(u)   \label{eq-relation4}\\
&& - \int_{\RR} \left(\phi_{V_{\infty}}(u(1+y)) - \phi_{V_{\infty}}(u) - u \phi_{V_\infty}'(u) y \mathds{1}_{|y|\leq 1}\right) \nu_U(dy) \nonumber\\
&=& - E \left[ e^{i u V_\infty} \psi_U ( u V_\infty) \right] \label{eq-relation3}
\end{eqnarray}
\end{theorem}
Equation \eqref{eq-relation3} can also be written in the compact
form
$$E \left[ (\psi_U(u V_\infty)+ \psi_L(u) ) e^{iu V_\infty} \right]
= 0 \quad \forall\; u \in \bR.$$

For the proof we need the following lemma. We use the notation
$S(\bR;\bC)$ to denote the class of complex valued functions
$f:\bR\to\bC$ such that $\Re(f) \in S(\bR)$ and $\Im(f) \in S(\bR)$.
Spaces like $C_c^\infty(\bR;\bC)$ are defined similarly. For a
generator $A$ we also write
$$D(A;\bC) := \{ f\in C_0(\bR;\bC) : \Re(f), \Im(f) \in D(A) \}.$$
It is clear that \eqref{generatorGOU-UL-ind} and hence
\eqref{eq-intgenerator} remain valid for complex valued functions
$f\in S(\bR;\bC)$.

\begin{lemma} \label{lemma-generatorlimit}
Let $(L_t)_{t\geq 0}$ be a L\'evy process in $\bR$ with generator
$A^L$, $\mu$ a fixed finite measure on $(\bR,\mathcal{B}_1)$ and
define $h,h_n,f$ and $f_n$ as in Theorem \ref{thm-generatorlimit}.
Then $f_n\in C_c^\infty(\RR;\bC)\subset D(A^L; \bC)$ and
$$\lim_{n\to\infty} \int_\RR A^Lf_n(x) \mu(dx) = \psi_L(u) \int_\RR e^{iux} \mu(dx) = \psi_L(u) \widehat{\mu} (u).$$
\end{lemma}
\begin{proof}
It is clear that $f_n\in C_c^\infty(\RR;\bC)$. From
\eqref{generatorlevy} we obtain
\begin{eqnarray*}
  \int_\RR A^Lf_n(x) \mu(dx)
&=& \gamma_L \int_{\RR} f_n'(x) \mu (dx) + \frac{\sigma_L^2}{2} \int_{\RR} f_n''(x) \,\mu(dx) \\
&&  + \int_{\bR} \int_{\RR}(f_n(x+y) - f_n(x) - f_n'(x)
y\mathds{1}_{|y|\leq 1}) \nu_L(dy)\,\mu(dx).
\end{eqnarray*} By
Taylor's formula, there are $\zeta_1, \zeta_2\in [-|y|,|y|]$ such
that
\begin{eqnarray*}
\lefteqn{\left| f_n(x+y) - f_n(x) - f_n'(x) y \mathds{1}_{|y|\leq 1}
\right|} \nonumber \\
& \leq & \left| f_n(x+y) - f_n(x) \right| \mathds{1}_{|y| > 1} +
2^{-1} \left[ |(\Re f_n'') (x+\zeta_1)| + |(\Im f_n'') (x+\zeta_2)| \right]   y^2 \mathds{1}_{|y|\leq 1} \nonumber \\
& \leq & 2 \| f_n\| \mathds{1}_{|y|> 1} + \|f_n''\| y^2
\mathds{1}_{|y|\leq 1} .
\end{eqnarray*}
Computing the first two derivatives of $f_n$ one easily sees that
they are uniformly bounded in $n$. Since additionally $\lim_{n\to
\infty} f_n'(x)=iue^{iux}$ and $\lim_{n\to \infty}
f_n''(x)=-u^2e^{iux}$ we obtain via dominated convergence
\begin{eqnarray*}
\lim_{n\to\infty} \int_{\bR} A^L f_n(x) \, \mu(dx) &=& \gamma_L
\int_{\RR} iue^{iux} \,\mu (dx) - \frac{\sigma_L^2}{2} \int_{\RR}
u^2 e^{iux}\,\mu(dx) \\
& & + \int_{\bR} \int_{\bR} \left( e^{iu(x+y)} -e^{iux} -
iue^{iux}y\mathds{1}_{|y|\leq 1} \right) \, \nu_L(dy) \,
\mu(dx) ,\\
\end{eqnarray*}
which gives the claim.
 \end{proof}

\begin{proof}[Proof of Theorem \ref{thm-generatorlimit}]
Since $\int_0^t \cE(U)_{s-} \, dL_s$ converges almost surely to the
finite random variable $V_\infty$ as $t\to\infty$, $\mu =
\cL(V_\infty)$ is the unique stationary marginal distribution and
hence invariant law of $V$ by \cite[Thms. 2.1 and
3.6]{behmelindnermaller11}. Equation \eqref{eq-relation1} then
follows directly from \eqref{generatorstat},
\eqref{generatorinteq} and Lemma \ref{lemma-generatorlimit}.\\
To show \eqref{eq-relation2}, observe that by Taylor's formula there
are $\zeta_1, \zeta_2\in [-|xy|,|xy|]$ such that
\begin{eqnarray*}
\lefteqn{\left|f_n(x+xy) - f_n(x) - xy f_n'(x) \mathds{1}_{|y|\leq
1} \right|} \nonumber \\
& \leq & \left| f_n(x+xy) - f_n(x)\right| \mathds{1}_{|y|>1} +
2^{-1} \left[ |(\Re f_n'') (x+\zeta_1)| + |(\Im f_n'') (x+\zeta_2)| \,\right]   x^2 y^2 \mathds{1}_{|y|\leq 1} \nonumber \\
& \leq & 2 \|f_n\| \mathds{1}_{|y|>1} + \|f_n''\| x^2 y^2
\mathds{1}_{|y|\leq 1}.
\end{eqnarray*}
Equation \eqref{eq-relation2} then follows directly from
\eqref{eq-relation1} by dominated convergence and Fubini's theorem,
observing as in the proof of Lemma \ref{lemma-generatorlimit} that
$f_n$ and its first two derivatives are uniformly bounded in $n$.
Finally, Equations \eqref{eq-relation4} and \eqref{eq-relation3} are
immediate consequences of \eqref{eq-relation2}.
\end{proof}

For GOU processes driven by $(\xi,\eta)$, Theorem
\ref{thm-generatorlimit} gives the following.

\begin{corollary} \label{cor-generatorlimit}
Let $(\xi_t)_{t\geq 0}$ and $(\eta_t)_{t\geq 0}$ be two independent
L\'evy processes such that $V_\infty=\int_0^\infty
e^{-\xi_{s-}}d\eta_s$ converges to a finite random variable.
Then $\mu=\cL(V_\infty)$ is the invariant law of the GOU process $(V_t)_{t\geq 0}$ driven by $(\xi, \eta)^T$ as defined in \eqref{GOUdef}.\\
Let $h\in C^\infty_c(\RR)$ such that $h(x)=1$ for $|x|\leq 1$ and
$h(x)=0$ for $|x|\geq 2$ and set $h_n(x):=h(\frac{x}{n})$ and
$f(x)=e^{iux}$, $f_n(x)=f(x)h_n(x)$ for $u\in \bR$. Then
\begin{eqnarray}
\psi_\eta(u)\phi_{V_{\infty}}(u)&=&\lim_{n\to \infty} \left(
\gamma_\xi \int_{\RR} xf_n'(x) \,\mu (dx) - \frac{\sigma_\xi^2}{2}
\int_{\RR} (x^2f_n''(x) + xf_n'(x) )\,\mu(dx) \right. \label{eq-relation5}\\
&& \left. - \int_\bR \int_{\RR} (f_n(xe^{-y}) - f_n(x) + xyf_n'(x)
\mathds{1}_{|y|\leq 1}) \nu_\xi(dy) \mu(dx)\right). \nonumber
\end{eqnarray}
If additionally $E [V_\infty^2] < \infty$, then
\begin{eqnarray}
\psi_\eta(u) \phi_{V_\infty} (u) & = & \gamma_\xi u
\phi_{V_\infty}'(u) - \frac{\sigma_\xi^2}{2} \left( u^2
\phi_{V_\infty}''(u) + u \phi_{V_\infty}'(u)\right) \label{eq-relation6} \\
& & - \int_{\bR} \left( \phi_{V_\infty} (u e^{-y}) -
\phi_{V_\infty}(u) + u y \phi_{V_\infty}'(u) \mathds{1}_{|y|\leq 1}
\right) \, \nu_\xi(dy). \nonumber
\end{eqnarray}
\end{corollary}

\begin{proof}
 This follows directly from Theorem \ref{thm-generatorlimit} and the relations between $(U,L)$ and
 $(\xi,\eta)$ as given in \eqref{eq-def-UL} and \cite[Lem.
 3.4]{behmelindnermaller11}, or alternatively using
 \eqref{generatorGOU-xieta-ind} and arguments as in the proof of
 Theorem \ref{thm-generatorlimit}.
\end{proof}

Observe that for $\xi$ being a compound Poisson process Equation \eqref{eq-relation6} immediately gives
\eqref{eq-motivation}.

\begin{remark}
Carmona \cite[Thm. 2]{carmona-unpublished} obtains a formula related
to \eqref{eq-relation6} under certain, more restrictive assumptions.
In particular, it is assumed in \cite{carmona-unpublished} that
$e^{\xi_t}$ admits a strictly positive density on some interval
$(0,r_t)$ for some $r_t > 0$. In the special case that $\eta$ is a
compound Poisson process without negative jumps and $\xi$ is a
Brownian motion with drift, formula \eqref{eq-relation6} has already
been obtained by Nilsen and Paulsen \cite[Prop. 2]{NilsenPaulsen},
stated for Laplace transforms.
\end{remark}

\begin{remark}
Let $\eta$ be a subordinator, $\xi$ a L\'evy process independent of
$\eta$, and suppose that $V_\infty := \int_0^\infty e^{-\xi_{s-}} \,
d\eta_s$ is almost surely finite. Then $V_\infty \geq 0$, and we can
also use Laplace transforms in the above derivation. More precisely,
let $(U,L)^T$ be given by \eqref{eq-def-UL}, so that $L=\eta$ by
independence and $e^{-\xi_t} = \cE(U)_t$, where $\nu_U((-\infty,-1])
= 0$. Denote the Laplace transforms of $\eta=L$ and $V_\infty$ for
$u\geq 0$ by $\LL_\eta(u) = \LL_L(u) = E [e^{-u \eta_1}] =
\phi_\eta(iu)$ and $\LL_{V_\infty} (u) = E [e^{-u V_\infty}]$,
respectively. Let $f$ be a function in $S(\RR)$ with $f(x)=e^{-ux},
x\geq 0$, then $f$ is in $D(A^V)$ for $u>0$ and a direct computation starting
from \eqref{eq-intgenerator} yields the following analogues of
\eqref{eq-relation2} and \eqref{eq-relation6} without any further
moment restrictions on the distribution of $V_\infty$.
 \begin{eqnarray*}
\lefteqn{\log \LL_L(u)= \log \LL_\eta(u)}\\
&=& u\gamma_U \frac{E\left[V_\infty e^{-uV_\infty} \right]}{\LL_{V_\infty}(u)}
-\frac{\sigma_U^2 u^2}{2} \frac{E\left[V_\infty^2 e^{-uV_\infty} \right]}{\LL_{V_\infty}(u)}\\
&& - \int_{(-1,\infty)}
\left(\frac{\LL_{V_\infty}(u(1+y))}{\LL_{V_\infty}(u)}
- 1+ u\frac{E\left[V_\infty e^{-uV_\infty} \right]}{\LL_{V_\infty}(u)} y \mathds{1}_{|y|\leq 1}\right) \nu_U(dy)\\
&=& -u\gamma_\xi \frac{E\left[V_\infty e^{-uV_\infty} \right]}{\LL_{V_\infty}(u)}
-\frac{\sigma_\xi^2 }{2} \left( \frac{E\left[V_\infty^2 e^{-uV_\infty} \right]}{\LL_{V_\infty}(u)}u^2 -
\frac{E\left[V_\infty e^{-uV_\infty} \right]}{\LL_{V_\infty}(u)} u \right)\\
&& -\int_{\RR} \left(\frac{\LL_{V_\infty}(ue^{-y})}{\LL_{V_\infty}(u)}
- 1- u\frac{E\left[V_\infty e^{-uV_\infty} \right]}{\LL_{V_\infty}(u)} y \mathds{1}_{|y|\leq 1}\right) \nu_\xi(dy), \quad u\geq 0.
\end{eqnarray*}
\end{remark}

The formula given in Corollary \ref{cor-generatorlimit} will be
useful in determining $\law(\eta_1)$ from $\law(V_\infty)$ and
$\law(\xi_1)$ as observed in Theorem \ref{thm-injectivity} below.
For the determination of $\law(\xi_1)$ from $\law(V_\infty)$ and
$\law(\eta_1)$, the following relation between the characteristic
triplets of $\xi$, $L$ and the characteristic function of $\log
|V_\infty|$ will be helpful.

\begin{theorem}\label{thm-generatorlimit-2}
 Let $(\xi_t)_{t\geq 0}$ and $(\eta_t)_{t\geq 0}$ be two independent L\'evy processes such that
$V_\infty=\int_0^\infty e^{-\xi_{s-}}d\eta_s$ converges to a finite
random variable and such that $\eta$ is not the zero process.
Then $\mu=\cL(V_\infty)$ is the invariant law of the GOU process $(V_t)_{t\geq 0}$ driven by $(\xi, \eta)^T$.\\
Let $h\in C^\infty_c(\RR)$ such that $h(x)=1$ for $|x|\leq 1$ and
$h(x)=0$ for $|x|\geq 2$ and set $h_n(x):=h(\frac{x}{n})$ and for
$x\neq 0$ and $u\in \bR$ define $f(x)=e^{iu\log |x| }$ and
$f_n(x)=e^{iu\log |x|}h_n(\log |x|)$ with $f_n(0) = 0$. Then
\begin{eqnarray}
\psi_{-\xi}(u)\phi_{\log |V_{\infty}|}(u)&=&- \lim_{n\to \infty} \int_{\RR} A^\eta f_n(x) \mu(dx)\label{eq-relation7}\\
&=&\lim_{n\to \infty} \left( -\gamma_\eta \int_{\RR} f_n'(x) \mu (dx) - \frac{\sigma_\eta^2}{2} \int_\RR f_n''(x) \mu(dx) \right. \label{eq-relation8}\\
&&\left. - \int_\bR \int_{\RR} (f_n(x+y) -
f_n(x)-f_n'(x)y\mathds{1}_{|y|\leq 1} \nu_\eta(dy) \mu(dx) \right).
\nonumber
\end{eqnarray}
If additionally $E [V_\infty^{-2}] < \infty$, then
\begin{eqnarray}
\lefteqn{\psi_{-\xi}(u)\phi_{\log |V_{\infty}|}(u)}
\label{eq-relation9} \\
&=& - iu \gamma_\eta E \left[ V_\infty^{-1}
e^{i u \log |V_\infty|} \right] + \frac{\sigma_\eta^2}{2} (iu + u^2)
E \left[ V_\infty^{-2} e^{iu \log
|V_\infty|} \right] \nonumber \\
& & - \int_{\bR} \left( E\left[ e^{iu \log |V_\infty + y|} \right] -
E \left[ e^{iu \log |V_\infty|} \right] - iuy E \left[ V_\infty^{-1}
e^{iu \log |V_\infty|} \right] \mathds{1}_{|y|\leq 1} \right) \,
\nu_\eta(dy). \nonumber
\end{eqnarray}
\end{theorem}
\begin{proof}
 Observe that obviously $f_n\in C_c^\infty(\RR;\bC)$ and thus $f_n\in D(A^V;\bC)\cap D(A^\eta;\bC)$.
On the other hand we obtain for $\tilde{f}(x)=f(e^x)$ and
$\tilde{f}_n(x)=f_n(e^x)$ that $\tilde{f}(x)= e^{iux}$ and
$\tilde{f}_n(x)=\tilde{f}(x)h_n(x)$ and hence $\tilde{f}_n\in
C_c^\infty(\bR;\bC)\subset D(A^{-\xi};\bC)$. Similarly for
$\tilde{\tilde{f}}(x)=f(-e^x)$ and
$\tilde{\tilde{f}}_n(x)=f_n(-e^x)$ we have $\tilde{\tilde{f}}(x)=
e^{iux}$ and $\tilde{\tilde{f}}_n(x)=\tilde{\tilde{f}}(x)h_n(x)$ and
also
$\tilde{\tilde{f}}_n\in C_c^\infty(\bR;\bC)\subset D(A^{-\xi};\bC)$. \\
Since $\mu(\{0\})=0$ by \cite[Thm. 2.2]{bertoinlindnermaller08}, we
obtain from \eqref{generatorGOU-xieta-ind-short} and
\eqref{generatorstat}
\begin{eqnarray*}
 0&=& \int_\RR A^Vf_n(x)\mu(dx)\\
&=& \int_{\RR} A^\eta f_n(x)\mu(dx)+ \int_{(0,\infty)} A^{-\xi} \tilde{f}_n(\log x)\mu(dx)
+ \int_{(-\infty,0)} A^{-\xi} \tilde{\tilde{f}}_n(\log |x|)\mu(dx).
\end{eqnarray*}
Setting $S_1: (0,\infty)\to \RR, x\mapsto \log x$, and $S_2:
(-\infty,0)\to \RR, x\mapsto \log (-x)$, we compute using Lemma
\ref{lemma-generatorlimit}
\begin{eqnarray*}
\lefteqn{\lim_{n\to\infty}\left(\int_{(0,\infty)} A^{-\xi} \tilde{f}_n(\log x)\mu(dx)+ \int_{(-\infty,0)} A^{-\xi} \tilde{\tilde{f}}_n(\log |x|)\mu(dx) \right)}\\
&=& \lim_{n\to\infty} \left( \int_{\RR} A^{-\xi} \tilde{f}_n(y) dS_1(\mu_{|(0,\infty)})(y)+ \int_{\RR} A^{-\xi} \tilde{\tilde{f}}_n(y)dS_2(\mu_{|(-\infty,0)})(y) \right)\\
&=& \psi_{-\xi}(u) \left(\int_{\RR} e^{iuy} dS_1(\mu_{|(0,\infty)})(y)+ \int_{\RR} e^{iuy} dS_2(\mu_{|(-\infty,0)})(y)\right)\\
&=& \psi_{-\xi}(u) \left(\int_{(0,\infty)} e^{iu\log x} \mu(dx)+ \int_{(-\infty,0)} e^{iu\log |x|} \mu(dx)\right)\\
&=& \psi_{-\xi}(u) \phi_{\log |V_{\infty}|}(u)
\end{eqnarray*}
which yields \eqref{eq-relation7} and \eqref{eq-relation8} via
\eqref{generatorlevy}.

Now assume that $E [V_\infty^{-2}] < \infty$. We have
$\tilde{f}_n(x) = e^{iux} h_n(x)$ and $f_n(x) = \tilde{f}_n(\log
|x|)$ for all $x\in \bR$. In particular, $f_n'(x) = x^{-1}
\tilde{f}_n'( \log |x|)$ and $f_n''(x) = x^{-2} ( \tilde{f}_n''(\log
|x|) - \tilde{f}_n' (\log |x|))$ for $x\neq 0$. For $|y|>1$, we
further have $|f_n(x+y) - f_n(x)| \leq 2 \|h\| < \infty$, and for
$|y|\leq 1$ such that $xy (x+y) \neq 0$ there are $\zeta_1,\ldots,
\zeta_4\in \bR$ by Taylor's theorem such that
\begin{eqnarray*}
\lefteqn{ \left| f_n(x+y) - f_n(x) - f_n'(x) y \right|} \\
& = & \left| \tilde{f}_n ( \log |x+y|) - \tilde{f}_n (\log |x|) -
\tilde{f}_n'( \log |x|)  yx^{-1} \right| \\
& \leq & \left| \tilde{f}_n ((\log |x|) + yx^{-1}) - \tilde{f}_n (
\log |x|) - \tilde{f}_n' (\log |x|) yx^{-1} \right| \\
& & + \left| \tilde{f}_n ( \log |x| + \log |1+ yx^{-1}|) -
\tilde{f}_n ((\log|x|) + yx^{-1}) \right| \\
& = & 2^{-1} \left| (\Re \tilde{f}_n'')((\log |x|) + \zeta_1) y^2
x^{-2}\right| + \left| (\Re \tilde{f}_n')((\log |x|) + yx^{-1} +
\zeta_2)  \left( \log |1 + yx^{-1}| - yx^{-1} \right) \right| \\
& & + 2^{-1} \left| (\Im \tilde{f}_n'')((\log |x|) + \zeta_3) y^2
x^{-2}\right| + \left| (\Im \tilde{f}_n')((\log |x|) + yx^{-1} +
\zeta_4)  \left( \log |1 + yx^{-1}| - yx^{-1} \right) \right| \\
& \leq &  \| \tilde{f}_n''\| y^2 x^{-2} + \| \tilde{f}_n'\| C y^2
x^{-2}
\end{eqnarray*}
for some universal constant $C$. Since $\|\tilde{f}_n\|$,
$\|\tilde{f}_n'\|$ and $\| \tilde{f}_n''\|$ are uniformly bounded in
$n$, since $\mu$ is continuous (cf. \cite[Thm.
2]{bertoinlindnermaller08}) so that $(\nu \otimes \mu) (\{(x,y)^T\in
\bR^2 : xy (x+y) = 0\} = 0$, since $\int_{\bR} x^{-2} \mu(dx) <
\infty$ by assumption and since $f_n$, $f_n'$ and $f_n''$ converge
on $\bR \setminus \{0\}$ to $f$, $f'$ and $f''$, respectively, by
dominated convergence the right hand side of \eqref{eq-relation8} is
equal to $$\int_{\bR} \left( - \gamma_\eta f'(x) -
\frac{\sigma_\eta^2}{2} f''(x) - \int_{\bR} \left(f(x+y) - f(x) -
f'(x) y \mathds{1}_{|y|\leq 1}\right) \, \nu_\eta(dy) \right)
\mu(dx),$$ which gives \eqref{eq-relation9}.
\end{proof}

\section{Injectivity} \label{sec-inj}
\setcounter{equation}{0}

Let $\xi=(\xi_t)_{t\geq 0}$ and $(\eta_t)_{t\geq 0}$ be two
independent L\'evy processes such that $V_\infty := \int_0^\infty
e^{-\xi_{s-}} \, d\eta_s$ converges almost surely. By \cite[Thm.
2]{ericksonmaller05}, this implies that $\xi$ drifts to $+\infty$.
As in the introduction, for a L\'evy process $(\xi_t)_{t\geq 0}$
such that $\xi_t \to +\infty$ a.s. as $t\to\infty$ denote
$$D_\xi := \{ \cL (\eta_1) : \eta \;\mbox{ L\'evy process, independent
of }\xi, \;\mbox{such that } \int_0^\infty e^{-\xi_{s-}} \, d\eta_s
\;\mbox{converges a.s.}\}$$ and consider the mapping
$$
\Phi_\xi : D_\xi   \to  \mathcal{P}(\bR) ,\quad  \cL(\eta_1) \mapsto
\cL \left( \int_0^\infty e^{-\xi_{s-}} \, d\eta_s \right), \quad
\mbox{where $\eta$ and $\xi$ are independent.}
$$
Here $\mathcal{P}(\bR)$ denotes the set of probability distributions
on $(\bR,\cB_1)$. For a L\'evy process $(\eta_t)_{t \geq 0}$ denote
further
$$\tilde{D}_\eta := \{ \cL (\xi_1) : \xi \;\mbox{ L\'evy process, independent
of }\eta, \;\mbox{such that } \int_0^\infty e^{-\xi_{s-}} \, d\eta_s
\;\mbox{converges a.s.}\}$$ and define $\tilde{\Phi}_\eta$ by
$$\tilde{\Phi}_\eta: \tilde{D}_\eta \to \mathcal{P}(\bR), \quad \law(\xi_1)
\mapsto  \cL \left( \int_0^\infty e^{-\xi_{s-}} \, d\eta_s \right),
\quad \mbox{where $\eta$ and $\xi$ are independent.}$$

We are interested in injectivity of the mappings $\Phi_\xi$ and
$\tilde{\Phi}_\eta$, or at least in injectivity of these mappings
when restricted to certain subsets. A key result for these
investigations will be the following theorem, which follows
immediately from \eqref{eq-relation5} and \eqref{eq-relation8}, by
dividing by $\phi_{V_\infty}(u)$ and $\phi_{\log |V_\infty|}(u)$
when different from zero, which is always the case for $u$ in a
neighborhood of zero.

\begin{theorem} \label{thm-injectivity}
Let $(\xi_t)_{t\geq 0}$ and $(\eta_t)_{t\geq 0}$ be two independent
L\'evy processes such that $V_\infty := \int_0^\infty e^{-\xi_{s-}}
\, d\eta_s$ converges almost surely. If $\phi_{V_\infty} (u) \neq 0$
for $u$ from a dense subset of $\bR$, or if $\law(\eta_1)$ is
uniquely determined by the values of its characteristic function in
a neighborhood of the origin, then $\law(\eta_1)$ is uniquely
determined by $\law(V_\infty)$ and $\law(\xi_1)$. Similarly, if
$\eta$ is not the zero process and $\phi_{\log |V_\infty|}(u) \neq
0$ for $u$ from a dense subset of $\bR$, or if $\law(\xi_1)$ is
uniquely determined by the values of its characteristic function in
a neighborhood of the origin, then $\law(\xi_1)$ is uniquely
determined by $\law(V_\infty)$ and $\law(\eta_1)$.
\end{theorem}

It is well known (e.g. \cite{loeve}) that not every distribution is
characterized by the values of its characteristic function in a
neighborhood of the origin. This remains true for infinitely
divisible distributions. To see this take two different
distributions $\mu_1$ and $\mu_2$ without atoms at $0$ whose characteristic functions
coincide in a neighborhood of the origin and consider the
corresponding compound Poisson distributions with L\'evy measures
$\mu_1$ and $\mu_2$. These are both infinitely divisible and their
characteristic functions $\exp(\hat{\mu}_i(u)-1)$ coincide in a
neighborhood of the origin.

We do not know if the characteristic
function of the stationary distribution of a GOU process cannot
vanish on a non-empty open interval. As shown by Il'inskii
\cite[Cor.~1]{Ilinskii}, a set $A \subset \bR$ is the zero set of
some characteristic function if and only if $A$ is closed, does not
contain 0 and is symmetric with respect to the origin. Hence, a
priori there is no reason why $\phi_{V_\infty}$ appearing in Theorem
\ref{thm-injectivity} should not vanish identically on some
interval.

Still, it is possible to give some sufficient conditions. We start
with the following lemma, which is a minor reformulation of results
in Kawata \cite{kawata1972} and Lucasz \cite{lucasz1970}.

\begin{lemma} \label{lem-zeros}
Let $X$ be a random variable with law $\mu$ and assume that there is
some $\varepsilon > 0$ such that $E e^{\varepsilon X} < \infty$ or
$E e^{-\varepsilon X} < \infty$. Then the characteristic function
$\phi_X=\widehat{\mu}$ cannot be identically zero on  non-empty open
intervals. Furthermore, if $Y$ is another random variable whose
characteristic function coincides with that of $X$ in a neighborhood
of 0, then $\law(Y) = \law(X) = \mu$.
\end{lemma}
\begin{proof}
Without loss of generality assume that $E e^{-\varepsilon X} <
\infty$. Then $g(z) := E e^{i z X}$ can be defined for all $z\in
\bC$ such that $0 \leq \Im z < \varepsilon$, it is continuous there
and analytic in $0 < \Im z < \varepsilon$. That $\phi_X$ cannot be
identically zero on non-empty open intervals then follows from
\cite[Cor. 1.14.1]{kawata1972}. Let $Y$ be another random variable
such that $\phi_Y(u) = \phi_X(u)$ for all $u \in (-a,a)$ with some
$a>0$. Since $\phi_Y(u) = \lim_{y\downarrow 0} g(u + iy)$ for $u\in
(-a,a)$, it follows from \cite[Thm. 11.1.1]{lucasz1970} and its
proof that $ E e^{- \varepsilon Y} < \infty$. That $\law(Y) =
\law(X)$ then follows from \cite[Thm. 9.6.2]{kawata1972}.
\end{proof}

Define $ID^{\rm sym}$ to be the set of all infinitely divisible
distributions $\law(\eta_1)$  which are symmetric, and $ID^{\rm
exp}$ to be the set of all infinitely divisible distributions whose
L\'evy measure $\nu_\eta$ has some one-sided exponential moment,
i.e. for which there is $\varepsilon > 0$ such that
$$\int_{1}^\infty e^{\varepsilon x} \, \nu_\eta(dx) < \infty \quad
\mbox{or} \quad \int_{-\infty}^{-1} e^{- \varepsilon x} \,
\nu_\eta(dx) < \infty.$$
Denote
$$D_\xi^{\rm sym, exp} := D_\xi \cap (ID^{\rm sym} \cup ID^{\rm
exp}), \quad D_\xi^{\rm sym} := D_\xi \cap ID^{\rm sym} \quad \mbox{and} \quad D_\xi^{\rm exp} := D_\xi \cap ID^{\rm
exp}.$$
With these notions, we get the following result:

\begin{theorem} \label{cor-phi-xi-injective}
Let $(\xi_t)_{t\geq 0}$ be a L\'evy process such that $\xi_t$
converges almost surely to $\infty$ as $t\to \infty$. Then
$(\Phi_\xi)|_{D_{\xi}^{\rm sym, exp}}$ is injective and $$\Phi_\xi
(D_\xi^{\rm sym, exp}) \cap \Phi_\xi ( D_\xi \setminus D_\xi^{\rm
sym, exp}) = \emptyset.$$ If additionally $\xi$ is spectrally
negative, or $\xi = q N$ for some constant $q>0$ and a Poisson
process $N$, then $\Phi_\xi$ is injective on $D_\xi$.
\end{theorem}

In the special case when $\xi_t = t$, we have a spectrally negative
$\xi$, and we recover the well known result (e.g. \cite[Prop.
3.6.10]{jurekmason}) that $\Phi_{\xi_t = t}$ is injective.

\begin{proof}[Proof of Theorem \ref{cor-phi-xi-injective}]
If $\xi$ is spectrally negative, then $V_\infty=\int_0^\infty
e^{-\xi_{t-}} \, d\eta_t$ is self-decomposable by Remark (ii) to
Theorem 2.2 in \cite{bertoinlindnermaller08}, hence infinitely
divisible so that $\phi_{V_\infty}(u) \neq 0$ for all $u\in \bR$.
Injectivity of $\Phi_\xi$ then follows from Theorem
\ref{thm-injectivity}. If $\xi = q N_t$ for $q>0$ and a Poisson
process $N$, then by Example \ref{example-start} we can write
$V_\infty = \sum_{i=0}^\infty e^{-qi} (\eta_{T_{i+1}}-\eta_{T_i}),$
where $(\eta_{T_{i+1}} - \eta_{T_i})_{i=0,1,2,\ldots}$ is i.i.d. and
infinitely divisible by \cite[Thm. 30.1]{sato}. Hence $V_\infty$ is
infinitely divisible, and injectivity of $\Phi_\xi$ follows from
Theorem \ref{thm-injectivity}.

Now let $\xi$ be an arbitrary L\'evy process drifting to infinity.
If $\law(\eta_1) \in D_\xi \cap ID^{\rm exp}$, then there is
$\varepsilon > 0$ such that $E e^{\varepsilon \eta_1} < \infty$ or
$E e^{-\varepsilon \eta_1} < \infty$ (cf. \cite[Thm. 25.17]{sato}),
and Theorem~\ref{thm-injectivity} and Lemma \ref{lem-zeros} show
that $(\Phi_\xi)|_{D_\xi^{\rm exp}}$ is injective and
$\Phi_\xi (D_\xi^{\rm exp}) \cap \Phi_\xi (D_\xi \setminus
D_\xi^{\rm exp}) = \emptyset$.

Finally, let $\xi$ be an arbitrary L\'evy process drifting to
infinity and $\law(\eta_1) \in D_\xi^{\rm sym}$. Conditioning on
$\xi$, for $f$ in the Skorokhod space $D([0,\infty), \bR)$ of
c\`adl\`ag functions, we have
$$\left( V_\infty | \xi = f \right) =
\int_0^\infty e^{-f(t)} \, d\eta_t,$$ which converges for
$P_\xi$-almost every $f$. For such $f$, $\int_0^\infty e^{-f(t)} \,
d\eta_t$ is infinitely divisible (e.g. Sato~\cite{sato2007}), and
hence $E (e^{iuV_\infty}|\xi = f) \neq 0$ for all $u\in \bR$. Since
$\int_0^\infty e^{-f(t)} \, d\eta_t$ is also symmetric, $E
(e^{iuV_\infty}|\xi = f)$ is real valued and continuous in $u$ and
hence strictly positive for all $u\in \bR$. It follows that
$$\phi_{V_\infty} (u) =
\int_{D([0,\infty),\bR)} E \left[ e^{iu V_\infty} |\xi= f \right] \,
P_\xi (df) > 0  \quad \forall u \in \bR.$$ Theorem
\ref{thm-injectivity} then shows that $(\Phi_\xi)|_{D_\xi^{\rm sym}}$ is injective and $\Phi_\xi (D_\xi^{\rm sym})
\cap \Phi_\xi (D_\xi \setminus D_\xi^{\rm sym}) =
\emptyset$. This finishes the proof.
\end{proof}

\begin{remark}
Theorem \ref{cor-phi-xi-injective} shows in particular that if $\xi$
is arbitrary (but drifting to $+\infty$), and $\eta$ is spectrally
positive or negative (which applies in particular if $\eta$ is a
subordinator or the negative of a subordinator), then the
distribution of $\eta_1$ is uniquely determined by $\law(V_\infty)$
and $\law(\xi_1)$.
\end{remark}

Let us now turn to injectivity properties of $\tilde{\Phi}_\eta$. We
start with the following lemma, which is immediate from Lemma
\ref{lem-zeros}.

\begin{lemma} \label{lem-zeros2}
Let $X$ be a random variable which has no atom at 0 and assume that
there is $\varepsilon > 0$ such that $E |X|^\varepsilon < \infty$ or
$E |X|^{-\varepsilon} < \infty$. Then the characteristic function
$\phi_{\log |X|}$ of $\log |X|$ cannot be identically zero on
non-empty open intervals.
\end{lemma}

Examples of random variables $X$ with  finite negative fractional
moment $E |X|^{-\varepsilon}<\infty$ are given by random variables
which have a density $f$ in a neighborhood of zero such that $f(x) =
O(x^{\alpha})$ as $|x|\to 0$ for some $\alpha > \varepsilon -1 $. In
particular, if $\law(X)$ is a self-decomposable non-degenerate
distribution, then $X$ has a density satisfying this condition for
some $\varepsilon > 0$, which follows from Theorems 28.4, 53.6 and
53.8 in \cite{sato}; observe that this is trivial if $X$ has a
non-zero Gaussian component.
 Hence, whenever $X\not\equiv 0$ is
self-decomposable, then $\phi_{\log |X|}$ cannot be identically zero
on non-empty open intervals.

Other examples are given in the next lemma, which shows that
$\int_0^\infty e^{-\xi_{s-}} \, d\eta_s$ will always have certain
negative fractional moments if $\eta$ is a subordinator with
strictly positive drift, or if $\eta$ has a non-trivial Brownian
motion component. This complements \cite[Lem. 2.1]{MaulikZwart} and
\cite[Lem. 3.3]{savovetal} who assume $\xi$ to have finite mean.

\begin{lemma} \label{lem-fractional-moment}
Let $\xi$ and $\eta$ be two independent L\'evy processes such that
$V_\infty = \int_0^\infty e^{-\xi_{s-}} \, d\eta_s$ converges almost
surely. Suppose that $\eta$ is a subordinator with strictly positive
drift, or that the Brownian motion part of $\eta$ is non-trivial
(i.e. $\sigma_\eta^2 > 0$). Then $E |V_\infty|^{-\varepsilon} <
\infty$ for every $\varepsilon \in [0,1)$. In the latter case (i.e.
when $\sigma_\eta^2 > 0)$, $V_\infty$ has a  bounded density on
$\bR$.
\end{lemma}

\begin{proof}
Suppose first that $\eta$ is a subordinator with strictly positive
drift $\gamma_\eta^0$. Let $\varepsilon \in (0,1)$. Define the
L\'evy process $\xi^\flat$ by $\xi_t^\flat = \xi_t - \sum_{0<s\leq
t, |\Delta \xi_s| > 1} \Delta \xi_s$. Let $\tau$ be the time of the
first jump of $\xi$ whose size is greater than $1$ in magnitude.
Then
$$V_\infty \geq \gamma_\eta^0 \int_0^{1\wedge \tau} e^{-\xi_{s-}} \,
ds \geq \gamma_\eta^0 (1\wedge \tau) \exp \left( - \sup_{0\leq s
\leq 1} |\xi_s^\flat|\right).$$ Since $\tau$ and $\xi^\flat$ are
independent and $\tau$ is exponentially distributed (or $\tau \equiv
\infty$), it follows $E (1\wedge \tau)^{-\varepsilon} < \infty$ and
$E \exp \left(\varepsilon \sup_{0\leq s \leq 1}
|\xi_s^\flat|\right) < \infty$ (cf. \cite[Thms. 25.17,
25.18]{sato}), so that $E V_\infty^{-\varepsilon} < \infty$ if
$\eta$ is a subordinator with strictly positive drift.

Now suppose that $\eta$ is a L\'evy process such that $\sigma_\eta^2
> 0$. Denote the Brownian motion component of $\eta$ by $B$, so that
$B$ and $\eta-B$ are independent. Then the conditional distribution
of $V_\infty$ given $\xi=f$ is given by $\int_0^\infty e^{-f(t-)} \,
dB_t + \int_0^\infty e^{-f(t-)} \, d(\eta_t - B_t)$. But
$\int_0^\infty e^{-f(t-)} \, dB_t$ is $N(0, \sigma_\eta^2
\int_0^\infty e^{-2f(s)} \, ds)$-distributed, hence its density is
bounded by $(2\pi \sigma_\eta^2\int_0^\infty e^{-2f(s)}\,
ds)^{-1/2}$. Hence also $(V_\infty|\xi = f)$ has a density, $g_f$
say, which is bounded by $(2\pi \sigma_\eta^2\int_0^\infty
e^{-2f(s)}\, ds)^{-1/2}$. It follows that $V_\infty$ has a density
given by $x\mapsto \int_{D([0,\infty))} g_f(x) \, P_\xi(df)$, and
since
$$\int_{D([0,\infty))} \left(2\pi \sigma_\eta^2 \int_0^\infty e^{-2f(s)} ds \right)^{-1/2} \,
P_\xi(df) = (2\pi\sigma_\eta^2)^{-1/2} E \left[ \left( \int_0^\infty
e^{-2\xi_{s-}} \, ds\right)^{-1/2}\right] < \infty$$ by the part
just proved, this density is bounded on $\bR$. This then also shows
that $E |V_\infty|^{-\varepsilon} < \infty$ for all $\varepsilon \in
[0,1)$.
\end{proof}

Recall that $ID^{\rm exp}$ denotes the set of all infinitely
divisible distributions whose L\'evy measure has some one-sided
exponential moment. Denote
$$\tilde{D}_\eta^{\rm exp} := \tilde{D}_\eta \cap ID^{\rm \exp}.$$
We can now prove the following injectivity result  regarding
$\tilde{\Phi}_\eta$:

\begin{theorem} \label{cor-eta-injective1}
Let $\eta=(\eta_t)_{t\geq 0}$ be a non-zero L\'evy process.  Then
$(\tilde{\Phi}_\eta)|_{\tilde{D}_\eta^{\rm exp}}$ is injective and
\begin{equation} \label{eq-injectivity1}
\tilde{\Phi}_\eta (\tilde{D}_\eta^{\rm exp}) \cap \tilde{\Phi}_\eta
(\tilde{D}_\eta \setminus \tilde{D}_\eta^{\rm exp}) = \emptyset.
\end{equation} If additionally $\eta$ is a subordinator with strictly positive
drift, or if the Brownian motion part of $\eta$ is non-trivial (i.e.
$\sigma_\eta^2 > 0$), or if $\eta$ is a compound Poisson process
without drift such that $\nu_\eta((-\infty,0)) = 0$ and $\int_0^1
x^{-\varepsilon} \, \nu_\eta(dx) < \infty$ for some $\varepsilon >
0$, then $\tilde{\Phi}_\eta$ is injective on $\tilde{D}_\eta$.
\end{theorem}

Observe that $\tilde{D}_\eta^{\rm exp}$ contains all $\law(\xi_1)
\in \tilde{D}_\eta$ such that $\xi$ is spectrally negative or
spectrally positive. In particular, subordinators are uniquely
determined by $\law(V_\infty)$ and $\law(\eta_1)$.

\begin{proof}[Proof of Theorem \ref{cor-eta-injective1}]
The injectivity of $\tilde{\Phi}_\eta$ on $\tilde{D}_\eta^{\rm exp}$
as well as \eqref{eq-injectivity1} are clear from
Theorem~\ref{thm-injectivity} and Lemma \ref{lem-zeros}. Similarly,
injectivity of $\tilde{\Phi}_\eta$ on $\tilde{D}_\eta$ follows from
Lemmas \ref{lem-zeros2}, \ref{lem-fractional-moment}  and
Theorem~\ref{thm-injectivity} if $\eta$ is a subordinator with
strictly positive drift or if $\sigma_\eta^2 > 0$.

Finally, let us prove injectivity of $\tilde{\Phi}_\eta$ when $\eta$
is a compound Poisson process with $\nu_\eta((-\infty,0)) = 0$ and
$\int_0^1 x^{-\varepsilon} \, \nu_\eta(dx) < \infty$ for some
$\varepsilon > 0$.
 Denote by $T$ the time of the
first jump of $\eta$. Then
$$V_\infty = \int_0^\infty e^{-\xi_{s-}} \, d\eta_s = e^{-\xi_{T-}} \Delta
\eta_T + e^{-\xi_T} \int_T^\infty e^{-(\xi_{s-} - \xi_T)} \, d\eta_s
= e^{-\xi_T} (\Delta \eta_T + V_\infty') \quad \mbox{a.s.},$$ since
$\xi$ and $\eta$ almost surely do not jump together. The random
variable $V_\infty'$ has the same distribution as $V_\infty$ and  is
independent of $(e^{-\xi_T}, \Delta \eta_T)$.  Observe further
that also $\xi_T$ and $\Delta \eta_T$ are independent. It follows
that
\begin{equation*}
\phi_{\log V_\infty}(u) = \phi_{-\xi_T}(u) \, \phi_{\log (\Delta
\eta_T + V_\infty')} (u), \quad u\in \bR.\end{equation*}
Since
$$E (\Delta \eta_T + V_\infty')^{-\varepsilon} \leq E (\Delta \eta_T)^{-\varepsilon} < \infty$$
as a consequence of $V_\infty' \geq 0$ and $\int_0^1
x^{-\varepsilon} \, \nu_\eta(dx) < \infty$, it follows from
Lemma~\ref{lem-zeros2} that $\phi_{\log (\Delta \eta_T +
V_\infty')}$ cannot vanish identically on non-empty open intervals.
Since $\phi_{-\xi_T}(u) \neq 0$ for all $u\in \bR$ as $\xi_T$ is
infinitely divisible,
it follows that  $\phi_{\log V_\infty}$ cannot vanish identically on
non-empty open intervals. Injectivity of $\tilde{\Phi}_\eta$ then
follows from Theorem \ref{thm-injectivity}.
\end{proof}

We do not know if $\Phi_\xi$ and $\tilde{\Phi}_\eta$ will always be
injective, but as we have seen in Theorems
\ref{cor-phi-xi-injective} and \ref{cor-eta-injective1},   the
mappings $\Phi_\xi$ and $\tilde{\Phi}_\eta$ are injective in many
cases. However, if we drop the condition of independence of $\xi$
and $\eta$, an injectivity result  does not hold, as shown in the
following. Therefore, additionally to the definitions at the
beginning of this section, for a L\'evy process $\xi$, let
$$D_\xi^{\rm dep} := \{ \cL (\chi_1,\eta_1) : (\chi,\eta) \;
\mbox{biv. LP such that $\int_0^\infty e^{-\chi_{s-}} \, d\eta_s$
converges a.s. and}\; \cL(\chi_1) = \cL(\xi_1)\}$$
and define the mapping
\begin{eqnarray*} \Phi_\xi^{\rm dep} : D_\xi^{\rm dep}  \to
\mathcal{P}(\bR) ,\quad  \cL(\chi_1,\eta_1)  \mapsto  \cL \left(
\int_0^\infty e^{-\chi_{s-}} \, d\eta_s \right).
\end{eqnarray*}
Then we obtain the following counterexample of injectivity.

\begin{example} \label{ex-1}
Let $\xi= N$ be a Poisson process. Then $\Phi_\xi^{\rm dep}$ is not
injective.
\end{example}

\begin{proof}
Let $(\chi,\eta)$ be a bivariate L\'evy process such that $\cL(
\chi_1,\eta_1) \in D_\xi^{\rm dep}$. By \cite[Thm.
2]{ericksonmaller05}, this means $\cL(\chi_1) = \cL(\xi_1)$ and
$E\log^+| \eta_1| < \infty$. Denote the time of the first jump of
$\chi$ by $T = T(\chi)$. Then
\begin{equation} \label{Al-12}
\int_0^\infty e^{-\chi_{t-}} \, d\eta_t = \eta_{T} + e^{-1}
\int_{T}^\infty e^{-(\chi_{t-} - \chi_{T})} \, d\eta_t.
\end{equation} Since $\int_{T}^\infty e^{-(\chi_{t-} - \chi_{T})} \,
d\eta_t$ has the same distribution as $\int_0^\infty e^{-\chi_{t-}}
\, d\eta_t =: W$, it follows that the characteristic function
$\phi_W$ of $W$ satisfies
$$\phi_W(x) = \prod_{k=0}^\infty \phi_{\eta_{T}} (e^{-k} x), \quad
x \in \bR$$ as shown in \cite{behme-et-al}. Thus, $\cL(W)$ is
determined by $\rho_{\chi,\eta} := \cL (\eta_{T})$ (not necessarily
vice versa!). Now let $(\chi^{(1)},\eta^{(1)}) \in D_\xi^{\rm dep}$
be such that $\eta^{(1)}$ is independent of $\chi^{(1)}$,
$\eta^{(1)}$ is not the zero process and $E \log^+
|\eta^{(1)}_{T(\chi^{(1)})}| < \infty$, and let
$(\chi^{(2)},\eta^{(2)})$ be a bivariate compound Poisson process
without drift and L\'evy measure
$$\nu_{\chi^{(2)},\eta^{(2)}} (dx,dy) = \delta_1(dx)
\rho_{\chi^{(1)}, \eta^{(1)}} (dy).$$ Then $(\chi^{(2)}, \eta^{(2)})
\in D_\xi^{\rm dep}$ and
$$\rho_{\chi^{(2)},\eta^{(2)}}= \cL( \eta^{(2)}_{T(\chi^{(2)})})= \rho_{\chi^{(1)},
\eta^{(1)}}.$$ It follows that both $(\chi^{(1)},\eta^{(1)})$ and
$(\chi^{(2)},\eta^{(2)})$ lead to the same distribution, giving an
example that injectivity is violated.
\end{proof}

\section{Ranges}\label{sec-ranges}
\setcounter{equation}{0}

The results of the previous section may now be used to determine
information on the ranges of the mappings $\Phi_\xi$ and
$\tilde{\Phi}_\eta$ as defined in Section \ref{sec-inj}. We start
with an elementary conclusion, which also follows from \cite[Thm.
2.2]{bertoinlindnermaller08} or \cite[Lem. 3.1]{behme2011}.

\begin{proposition}
 Let $\xi$ be non-deterministic, then  $\Phi_\xi (D_\xi \setminus \{ \law(0)\})$ is a subset of the continuous distributions.
Analoguously, if $\eta$ is non-deterministic, then the range of
$\tilde{\Phi}_\eta$ is a subset of  the continuous distributions.
\end{proposition}
\begin{proof}
 It follows directly from \cite[Thm. 1.3]{alsmeyeretal} that the distribution of the treated exponential functional
fulfills a pure type theorem, in particular it is either continuous,
or a Dirac measure. Suppose that $L_1=\eta_1\not\equiv 0 $.
Inserting the characteristic function $\phi(u)=e^{iuk}$, $k\in\bR$,
of a Dirac measure in \eqref{eq-relation3}, one immediately obtains
$\psi_L(u)=-\psi_U(ku)$ which can only hold for deterministic
processes $L_t=-kU_t=\gamma_L t$ with $k\neq 0$ and hence
deterministic $\eta$ and $\xi$.
\end{proof}

Recall the definition of $\Phi_\xi^{\rm dep}$ from the previous
section. Also recall that a distribution $\mu$ on $(\bR,\cB_1)$ is
called {\it $b$-decomposable}, where $b\in (0,1)$, if there exists a
probability measure $\rho$ on $(\bR,\cB_1)$ such that
$\widehat{\mu}(z) = \widehat{\mu}(bz) \widehat{\rho}(z)$ for all
$z\in \bR$.

\begin{proposition} \label{prop-range1}
Let $\xi=N$ be a Poisson process. Then the range of $\Phi_{\xi}^{\rm
dep}$ is the class of all $e^{-1}$-decomposable distributions.
\end{proposition}

\begin{proof}
That all distributions in the range of $\Phi_{\xi}^{\rm dep}$ are
$e^{-1}$-decomposable is clear from \eqref{Al-12}. Conversely, let
$\cL(W)$ be an $e^{-1}$-decomposable distribution. Then there exists
an i.i.d. noise sequence $(Z_n)_{n\in \bN_0}$ such that
\begin{equation} \label{eq-range1}
\sum_{k=0}^n e^{-k} Z_k \stackrel{d}{\to} W, \quad n\to\infty,
\end{equation}
which follows by iterating the defining equation $W \stackrel{d}{=}
e^{-1} W' + Z$ with $W'$ independent of $Z$ for
$e^{-1}$-decomposability. Hence $\sum_{k=0}^n e^{-k} Z_k$ converges
in distribution and hence almost surely as $n\to\infty$ and the
Borel-Cantelli-lemma implies that $Z_0$ must have finite
$\mbox{log}^+$-moment. Now define the compound Poisson process
$(\chi,\eta)$ without drift and L\'evy measure
$$\nu_{\chi,\eta} (dx,dy) = \delta_1 (dx) \cL(Z_0)(dy).$$
Then $\cL(\chi,\eta) \in D_\xi^{\rm dep}$ (due to the finite log$^+$
moment of $Z_0$), and with the notations of Example \ref{ex-1} it
follows that $\cL(\eta_{T(\chi)}) = \cL(Z_0)$. Hence $\Phi_\xi^{\rm
dep} (\cL(\chi,\eta)) = \cL(W)$.
\end{proof}

\begin{proposition}
 Let $\xi=N$ be a Poisson process. Then the range of $\Phi_{\xi}$ is a subset of the class of infinitely divisible $e^{-1}$-decomposable distributions
 without Gaussian part.
\end{proposition}
\begin{proof}
 By Proposition \ref{prop-range1} it remains to show that $W=\int_{(0,\infty)} e^{-N_{s-}}d\eta_s$ is infinitely divisible and has zero Gaussian part. Therefore
denote the time of the first jump of $N$ by $T$, then $Z_0:=\eta_T$
is infinitely divisible without Gaussian part as a consequence of
\cite[Thm. 30.1]{sato}. Hence by \eqref{eq-range1} also $W$ is infinitely divisible and the Gaussian part of $W$ is zero.
\end{proof}

It is well known that the OU process is a Gaussian process whose
stationary distribution is normally distributed. In particular
$\int_{(0,\infty)} e^{-t\sigma^2/2}d(\sigma W_t)$ for $W_t$ a
standard Brownian motion (Wiener process) is standard normally
distributed. The following theorem shows that this is the only
possible choice of $(\xi,\eta)^T$ which leads to a centered normal
distribution.

\begin{theorem} \label{thm-nonormal}
Let $\xi$ and $\eta$ be two independent L\'evy processes such that
$\int_0^\infty e^{-\xi_{s-}} \, d\eta_s$ converges almost surely.
Let $v> 0$. Then $\law(V_\infty) = N(0,v^2)$ if and only if there is
$\gamma_\xi > 0$ such that $\xi_t = \gamma_\xi t$ and $\eta_t =
(2\gamma_\xi)^{1/2} v W_t$, where $(W_t)_{t\geq 0}$ is a standard
Brownian motion.
\end{theorem}

\begin{proof}
That $\law(V_\infty) = N(0,v^2)$ if $\xi$ and $\eta$ are as
described is well known and follows as discussed above. Let us show
the converse and assume that $V_\infty$ is $N(0,v^2)$-distributed.
By replacing $\eta$ by $v^{-1}\eta$ we may assume that $v=1$. Inserting
$\phi_{V_\infty}(u) = e^{-u^2/2}$ in \eqref{eq-relation6}, we obtain
for $u\in \bR$
\begin{equation} \label{eq-relation10}
\psi_\eta(u) = -\gamma_\xi u^2 - \sigma_\xi^2 (u^4-2u^2)/2 -
\int_{\bR} \left( e^{-u^2 (e^{-2y} -1)/2} -1 - u^2 y
\mathds{1}_{|y|\leq 1} \right)\, \nu_\xi(dy).
\end{equation}
For given $u\in \bR$ denote $$f_u(y) := e^{-u^2 (e^{-2y} -1)/2} -1 -
u^2 y \mathds{1}_{|y|\leq 1}, \quad y \in \bR\setminus \{0\}.$$
 We shall first investigate the limit behavior of
\eqref{eq-relation10} as $u\to \infty$ when divided by appropriate
powers of $u$ and from that obtain information about the
characteristic triplet of $\xi$. To do so, observe first that there
are constants $C_1, C_2, C_3>0$ such that
\begin{eqnarray*}
|e^{-x} - 1 + x - x^2/2| & \leq & C_1 x^2 \quad \forall\; x > 0,\\
|e^{-2y }- 1 + 2y| & \leq & C_2 y^2 \quad \forall\; y \in [-1,1],
\quad \mbox{and} \\
(e^{-2y} - 1)^2 & \leq & C_3 y^2 \quad \forall\; y \in [-1,1].
\end{eqnarray*}
Let $y_0 \in [-1,0)$. Then  $|f_u(y)| \leq 1+ u^2$ for $y< y_0$, and
for $y \in [y_0,0)$ we can estimate
\begin{eqnarray*}
|f_u(y)| & \leq &  \left| -u^2 (e^{-2y}-1)/2 + u^4 (e^{-2y} -1)^2/8
- u^2
y\right| + C_1 u^4 (e^{-2y} -1)^2 /4 \\
& \leq & u^2 C_2 y^2/2 + u^4 C_3 y^2/8 + C_1 C_3 u^4y^2/4.
\end{eqnarray*}
Using dominated convergence, this gives
$$\limsup_{u\to \infty} u^{-4} \int_{-\infty}^0 |f_u(y)| \,
\nu_\xi(dy) \leq (C_3/8 + C_1 C_3/4) \int_{[y_0, 0)} y^2\,
\nu_\xi(dy),$$ and letting $y_0 \uparrow 0$ we see that
\begin{equation} \label{eq-relation11}
\lim_{u\to\infty} u^{-4} \int_{-\infty}^0 |f_u(y)| \, \nu_\xi(dy) =
0.
\end{equation}
Now let $y>0$. Then
$$f_u(y) \geq \left( u^2 (1-e^{-2y})/2 - u^2 y\right)
\mathds{1}_{(0,1]}(y) \geq -u^2 (C_2/2) y^2 \mathds{1}_{(0,1]}(y).$$
Since also $\lim_{u\to\infty} u^{-5} f_u(y) = +\infty$ for $y>0$ and $\lim_{u\to \infty} \int_{(0,1]} u^{-3} y^2 \, \nu_\xi(dy) = 0$, we obtain from Fatou's lemma
\begin{eqnarray*}
\lefteqn{\liminf_{u\to\infty} u^{-5} \int_{(0,\infty)}
f_u(y) \, \nu_\xi(dy) }\\
& = & \liminf_{u\to\infty} \int_{(0,\infty)} u^{-5} \left( f_u(y) + u^2 (C_2/2) y^2 \mathds{1}_{(0,1]} (y) \right) \, \nu_\xi(dy) \\
& \geq & \int_{(0,\infty)} \liminf_{u\to\infty}  \left( u^{-5} f_u(y) + u^{-3} (C_2/2) y^2 \mathds{1}_{(0,1]} (y) \right) \, \nu_\xi(dy) \\
& = & \int_{(0,\infty)} \infty \, \nu_\xi(dy) =
\infty \,  \nu_\xi((0,\infty)).
 \end{eqnarray*}
Dividing \eqref{eq-relation10} by $u^5$ and observing that
$\lim_{u\to\infty} u^{-2} \psi_\eta(u) = -\sigma_\eta^2 /2 < \infty$
(cf. \cite[Lem. 43.11]{sato}) and hence $\lim_{u\to\infty} u^{-5} \psi_\eta(u) = 0$, this together with
\eqref{eq-relation11} gives $\nu_\xi((0,\infty)) = 0$. Similarly,
dividing \eqref{eq-relation10} by $u^4$, we obtain $\sigma_\xi^2 =
0$ by \eqref{eq-relation11}.

It remains to show that $\nu_\xi((-\infty,0)) = 0$. In doing so, we
shall first establish that $\xi$ must be of finite variation. Recall
that
\begin{eqnarray*}
e^{-x} - 1 + x & \geq & 0 \quad \forall\; x \geq 0 \quad
\mbox{and}\\
e^{-x} - 1 + x & \geq & x/2 \quad \forall\; x \geq 4.
\end{eqnarray*}
Let $y <0$. Then $f_u(y) \geq -1$ for $y<  -1$, and for $y\in
[-1,0)$ we estimate
\begin{eqnarray*}
f_u(y) & = & e^{-u^2 (e^{-2y}-1)/2} - 1 + u^2 (e^{-2y}-1)/2 - u^2
(e^{-2y}-1) / 2 - u^2 y \\
& \geq & u^2 (e^{-2y}-1)/4 \, \mathds{1}_{ \{u^2 (e^{-2y}-1)/2 \geq
4\} } - C_2 u^2 y^2/2.
\end{eqnarray*}
An application of Fatou's lemma then shows
$$\liminf_{u\to \infty} u^{-2} \int_{(-\infty, 0)} f_u(y) \,
\nu_\xi(dy) \geq -C_2 / 2 \, \int_{[-1,0)} y^2 \, \nu_\xi(dy) +
\int_{[-1,0)} (e^{-2y}-1)/4 \, \nu_\xi(dy).$$ But since
$\lim_{u\to\infty} u^{-2} |\psi_\eta (u)| < \infty$, dividing
\eqref{eq-relation10} by $u^2$ and letting $u\to\infty$ gives
$\int_{[-1,0)} (e^{-2y}-1) \, \nu_\xi(dy) < \infty$, hence
$\int_{[-1,0)} |y| \, \nu_\xi (dy) < \infty$, so that $\xi$ is of
finite variation. Equation \eqref{eq-relation10} can now be
rewritten as
\begin{equation}\label{eq-relation12}
\psi_\eta(u) + \int_{(-\infty,0)} \left( e^{-u^2 (e^{-2y}-1)/2}
-1\right) \, \nu_\xi(dy) = -\gamma_\xi^0 u^2,
\end{equation}
where $\gamma_\xi^0$ is the drift of $\xi$. Since $\xi_t\to \infty$
as $t\to\infty$ and $\xi$ is spectrally negative, we must have
$\gamma_\xi^0 > 0$.

Let $\rho$ denote the standard normal distribution and define the
mapping $T$ by
$$T: \bR \times (-\infty,0) \to \bR, \quad (x,y) \mapsto x
\sqrt{e^{-2y}-1}.$$ Then for any $\varepsilon > 0$,
\begin{eqnarray*}
\lefteqn{\int_{(-\infty,-\varepsilon]} \left( e^{-u^2 (e^{-2y}-1)/2}
- 1 \right) \, \nu_\xi(dy)} \\
 & = & \int_{\bR} \int_{\bR}
\left(e^{iu x \sqrt{e^{-2y}-1} }- 1\right) \, \rho(dx) \,
{\nu_\xi}|_{(-\infty,-\varepsilon]}(dy) \\
& = & \int_{\bR} (e^{iuz} - 1) \,  T(\rho \otimes
{\nu_\xi}|_{(-\infty,-\varepsilon]}) (dz).
\end{eqnarray*}
With $$\overline{\psi}_\varepsilon (u) := \psi_\eta(u) + \int_{\bR}
(e^{iuz} - 1) \,  T(\rho \otimes
{\nu_\xi}|_{(-\infty,-\varepsilon]}) (dz)$$ it follows from
\eqref{eq-relation12} that $\lim_{\varepsilon \downarrow 0}
\overline{\psi}_\varepsilon (u) = -\gamma_\xi^0 u^2$. But since
$\overline{\psi}_\varepsilon$ is the L\'evy-Khintchine exponent of
an infinitely divisible distribution with L\'evy measure $\nu_\eta +
T(\rho \otimes {\nu_\xi}|_{(-\infty,-\varepsilon]})$, since $T(\rho
\otimes {\nu_\xi}|_{(-\infty,-\varepsilon]})$ is increasing as
$\varepsilon \downarrow 0$, and since $u\mapsto -\gamma_\xi^0 u^2$
is the L\'evy-Khintchine exponent of a Gaussian random variable, it
follows from \cite[Thm. 8.7]{sato} that $T(\rho \otimes
{\nu_\xi}|_{(-\infty,-\varepsilon]}) = 0$ for any $\varepsilon > 0$,
hence $\nu_\xi ((-\infty,0)) = 0$.

We have shown that $\xi_t = \gamma_\xi^0 t$. Injectivity of  the
mapping $\Phi_\xi$ (cf. Theorem~\ref{cor-phi-xi-injective}) together
with the sufficiency part show that necessarily $\eta_t =
(2\gamma_\xi)^{1/2} v W_t$, completing the proof.
\end{proof}

\section{Continuity}\setcounter{equation}{0}
\label{sec-cont}

Another natural question about the mappings $\Phi_\xi$ and $\tilde{\Phi}_\eta$ as defined in Section \ref{sec-inj} is, whether
they are continuous. Hereby we say, that $\Phi_\xi$ is continuous, if for each sequence of L\'evy processes $(\eta^{(n)})_{n\in\NN}$
 such that  $\eta_1^{(n)}\overset{d}\to \eta_1$ as $n\to\infty$ and $\cL(\eta_1^{(n)})\in D_\xi$, $\cL(\eta_1)\in D_\xi$, the sequence
$\Phi_\xi(\cL(\eta_1^{(n)}))$ converges weakly to $\Phi_\xi(\cL(\eta_1))$ as $n\to\infty$, denoted as $\Phi_\xi(\cL(\eta_1^{(n)})) \overset{w}\to
\Phi_\xi(\cL(\eta_1))$ in the following. Continuity of $\tilde{\Phi}_\eta$ is defined similarly. \\
In general $\Phi_\xi$ is not continuous as proven by the following
counterexample. We expect that failure of continuity of
$\Phi_{\xi_t=t}$ is known as it is a very well studied mapping, but
since we were unable to find a ready reference we give a short
proof.

\begin{example} \label{ex-2}
Let $(\xi_t=t)_{t\geq 0}$ be deterministic. Then $\Phi_\xi$ is not continuous.
\end{example}
\begin{proof}
In the given setting we have that $D_\xi$ is $ID_{\log}$, the set of
infinitely divisible distributions with finite log$^+$-moment. Now let
$(Y_i^{(n)})_{i\in\NN}$ be sequences of i.i.d. random variables such
that
$$\nu^{(n)}:=\cL(Y_1^{(n)})=(1-\frac 1 n )\left( \frac 1 2 \delta_1 + \frac 1
2 \delta_{-1}\right) + \frac 1 n \left( \frac 1 2 \delta_{n^n} +
\frac 1 2 \delta_{-n^n}\right)$$ and define the sequence
$(Y_i^{(0)})_{i\in\NN}$ of i.i.d. random variables with
$$\nu^{(0)}:=\cL(Y_1^{(0)})=\left( \frac 1 2 \delta_1 + \frac 1 2
\delta_{-1}\right). $$ Then obviously we have
$Y_i^{(n)}\overset{d}\to Y_i^{(0)}$ as $n\to \infty$. Now for all
$n\in\NN_0$ define the compound Poisson process
$\eta_t^{(n)}:=\sum_{i=1}^{N_t} Y_i^{(n)}$ where $N$ is a Poisson
process with rate $1$, independent of $(Y_i^{(n)})_{i\in \bN}$. Then
$\mu^{(n)}:=\cL(\eta_1^{(n)})\in D_\xi$ for all $n\in \NN_0$ and in
particular for $n\geq 1$ and $z\in\RR$ we have that
\begin{eqnarray*}
\widehat{\mu^{(n)}}(z)&=&\exp\left(\int_{\RR} (e^{izx}-1)\,\nu^{(n)}(dx)\right)= \exp(\widehat{\nu^{(n)}}(z)-1)
\\&\overset{n\to\infty}\to& \exp(\widehat{\nu^{(0)}}(z)-1) = \widehat{\mu^{(0)}}(z)
\end{eqnarray*}
such that $\mu^{(n)}\to \mu^{(0)}$ as $n\to \infty$. But
$\phi_\xi(\mu^{(n)})$ does not converge to $\phi_\xi(\mu^{(0)})$ as
will be shown in the following. Herefore observe that by \cite[Eq.
(17.14)]{sato} the L\'evy measure $\tilde{\nu}^{(n)}$ of
$\phi_\xi(\mu^{(n)})$ fulfills for all $n\geq 0$
\begin{equation*}
\tilde{\nu}^{(n)}([1,\infty))
= \int_\RR \int_0^\infty \mathds{1}_{[1,\infty)}(e^{-s}y) ds \,\nu^{(n)}(dy)
= \int_{(0,\infty)} \log y \,\nu^{(n)}(dy)
\end{equation*}
such that for all $n\geq 1$
$$\tilde{\nu}^{(n)}([1,\infty))= \frac{1}{2} \log n \to \infty \mbox{ as } n\to \infty,$$
whereas $\tilde{\nu}^{(0)}([1,\infty))= 0$. Using \cite[Thm.
8.7]{sato} this shows that $\Phi_\xi(\mu^{(n)})
\not\overset{w}\to \Phi_\xi(\mu)$ as $n\to\infty$, so that
$\Phi_\xi$ is not continuous.
\end{proof}

Continuity of stationary solutions of random recurrence equations
has been studied by Brandt \cite{brandt}. The following is a
special case of his result for i.i.d. sequences, but does not assume
that  $E [\log |B_0^{(n)}|]$, $E [|\log B_0|]$ are finite and that
$E [\log |B_0^{(n)}|] \to E [|\log |B_0|]$ as $n\to\infty$. That
these conditions can be omitted follows readily by an inspection of
Brandt's proof \cite[Thm. 2]{brandt}.

\begin{proposition} \label{prop-brandt}
 Let the sequences $(A_i ,B_i)_{i\in\NN_0}$, $(A_i^{(1)} ,B_i^{(1)})_{i\in\NN_0}$ $(A_i^{(2)} ,B_i^{(2)})_{i\in\NN_0}$, $\ldots$ be i.i.d. such that
$E [\log^+|A_0^{(n)}|] < \infty$, $E [\log^+|B_0^{(n)}|]<\infty$ for
all $n$, $E[\log^+|A_0|] < \infty$ and $E [\log^+|B_0|] < \infty$.
Assume further that $$-\infty< E[\log |A_0^{(n)}|]<0 \quad \mbox{for
all $n$,} \quad \quad -\infty< E[\log |A_0|]<0$$ and that for $n\to
\infty$
\begin{eqnarray*}
 (A_0^{(n)} ,B_0^{(n)}) &\overset{d}\to& (A_0 ,B_0),\\
 E[\log^+ |A_0^{(n)}|] &\to& E[\log^+ |A_0|],\\
  E[\log^+ |B_0^{(n)}|] &\to& E[\log^+ |B_0|] \\
\mbox{and }E[\log |A_0^{(n)}|] &\to& E[\log |A_0|].
\end{eqnarray*}
Let $Y^{(n)}_\infty$ be the unique stationary marginal distribution
of the random recurrence equation $Y^{(n)}_{i+1}=A^{(n)}_i
Y_i^{(n)}+ B_i^{(n)}$, $i\in\NN_0$, and define $Y_\infty$
analoguously. Then
$$(A_0^{(n)} ,B_0^{(n)}, Y^{(n)}_\infty) \overset{d}\to (A_0 ,B_0, Y_\infty)\quad \mbox{as}\quad n\to \infty$$
such that in particular $Y^{(n)}_\infty \overset{d}\to Y_\infty$ as
$n\to\infty$.
\end{proposition}

Due to the fact that generalized Ornstein-Uhlenbeck processes are
the conti\-nuous-time analogon of the solutions to random recurrence
equations with i.i.d. coefficients, we can use the above proposition
in our setting to obtain the following.

\begin{theorem} \label{thm-continuity}
 Let $(\xi^{(n)}, \eta^{(n)})$, $n\in\bN$, and $(\xi,\eta)$ be
bivariate L\'evy processes such that
$$(\xi_1^{(n)}, \eta_1^{(n)}) \stackrel{d}{\to} (\xi_1,\eta_1),
\quad n\to\infty.$$ Suppose there exists $\delta>0$ such that
\begin{eqnarray}
\label{contcond6} \sup_{n\in \bN} \int_{\bR \setminus [-1,1]}
(\log^+
|x|)^{1+\delta} \, \nu_{\eta^{(n)}}(dx) &<& \infty\\
 \label{contcond1}
\mbox{and}\quad  \sup_{n\in \bN}
E[|\xi_1^{(n)}|^{1+\delta}]&<&\infty.
\end{eqnarray}
Then $E \log^+ |\eta_1| < \infty$ and $E |\xi_1| < \infty$. Assume
further that
\begin{equation} \label{contcond3}
 E \xi_1>0 \quad \mbox{and}\quad  E \xi_1^{(n)}>0, \; n\in \bN.
\end{equation}
Then  $\int_0^\infty e^{-\xi_{s-}^{(n)}} \, d\eta_s^{(n)}$ converges
almost surely absolutely for each $n\in \bN$, as does $\int_0^\infty
e^{-\xi_{s-}} \, d\eta_s$, and
\begin{equation} \label{eq-8.16}
\int_0^\infty e^{-\xi_{s-}^{(n)}} \, d\eta_s^{(n)} \stackrel{d}{\to}
\int_0^\infty e^{-\xi_{s-}} \, d\eta_s, \quad n\to\infty.
\end{equation}
\end{theorem}

For the proof of Theorem \ref{thm-continuity} we need the following Lemma, which is of its own interest.

\begin{lemma} \label{lem-8.5}
Let $L = (L_t)_{t\geq 0}$ be a  L\'evy process in $\bR$ with
characteristic triplet $(\gamma_L, \sigma_L^2, \nu_L)$. Let $b>0$. Then
there exist universal constants $C_1, C_2, C_3 \in (0,\infty)$,
depending only on $b$, such that for every adapted c\`adl\`ag
process $H$ satisfying
$$E \left( \log^+ \sup_{0\leq s \leq 1} |H_s| \right)^b < \infty$$
the following estimate holds:
\begin{eqnarray}
\lefteqn{ E \left( \log^+ \sup_{0<s\leq 1} \left| \int_0^s H_{u-} \,
dL_u \right| \right)^b } \nonumber \\
& \leq &   C_1 \left( 1 + \sigma^2_L + \int_{|x|\leq 1} x^2 \, \nu_L(dx) +
\log^+ |\gamma_L| + \exp \left\{ C_2 \int_{|x|>1} (\log^+ |x|)^b \,
\nu_L(dx) \right\} \right) \nonumber \\
& & +  C_3 \, E\left( \log^+ \sup_{0\leq s\leq 1} |H_s| \right)^b .
\label{eq-8.12}
\end{eqnarray}
\end{lemma}

\begin{proof}
Write $L_t = L_t^\sharp + L_t^\flat$, where $L^\sharp =
(L_t^\sharp)_{t\geq 0}$ has characteristic triplet
$$(\gamma_L^\sharp := 0, (\sigma_L^\sharp)^2 := \sigma^2_L, \nu_L^\sharp := {\nu_L}|_{[-1,1]})$$ and $L^\flat = (L^\flat_t)_{t\geq 0}$ has characteristic triplet
$$(\gamma_L^\flat := \gamma_L, (\sigma_L^\flat)^2 := 0 , \nu_L^\flat := {\nu_L|}_{ \bR \setminus
[-1,1]}).$$ Then $L^\sharp$ has expectation zero (e.g.~\cite[Ex. 25.12]{sato}) and is a square integrable martingale, and $L^\flat$ is a
compound Poisson process together with drift $\gamma_L$. Observe
that for proving \eqref{eq-8.12} it is obviously sufficient to prove
it for $L^\sharp$ and $L^\flat$ separately, which we shall do.

For the estimate for $L^\sharp$, let $x>0$. Then
\begin{eqnarray}
\lefteqn{ P \left( \left( \log^+ \sup_{0<s\leq 1} \left| \int_0^s
H_{u-} \, dL_u^\sharp \right| \right)^b > x \right) } \nonumber \\
& = & P \left( \sup_{0<s\leq 1} \left| \int_0^s H_{u-} \,
dL_u^\sharp \right| > \exp (x^{1/b}) \right) \nonumber \\
& \leq & P \left( \sup_{0<s\leq 1}  \left| \int_0^s H_{u-} \,
dL_u^\sharp \right| > \exp (x^{1/b}), \, \sup_{0\leq s \leq 1} |H_s|
\leq \exp (x^{1/b}/2) \right) \nonumber \\
& & + P \left( \sup_{0\leq s \leq 1} |H_s|
> \exp (x^{1/b} / 2) \right). \label{eq-8.13}
\end{eqnarray}
Denote $H_s^{(x)} := H_s \wedge \exp (x^{1/b}/2)$. Then on
$\{\sup_{0\leq s\leq 1} |H_s| \leq \exp (x^{1/b}/2)\}$, $\int_0^s
H_{u-} \, dL_u^\sharp = \int_0^s H_{u-}^{(x)} \, dL_u^\sharp$ for
all $0\leq s \leq 1$, so that by Markov's inequality and Doob's
maximal quadratic inequality, we obtain
\begin{eqnarray}
\lefteqn{P \left( \sup_{0<s\leq 1}  \left| \int_0^s H_{u-} \,
dL_u^\sharp \right| > \exp (x^{1/b}), \, \sup_{0\leq s \leq 1} |H_s|
\leq \exp (x^{1/b}/2) \right)} \nonumber \\
& \leq & P \left( \sup_{0 \leq s \leq 1} \left| \int_0^1
H_{u-}^{(x)} \, dL_u^\sharp\right| > \exp (x^{1/b}) \right) \nonumber \\
& \leq & \exp (- 2 x^{1/b}) \, E \sup_{0\leq s \leq 1} \left|
\int_0^1 H_{u-}^{(x)} \, dL_u^\sharp \right|^2 \nonumber \\
& \leq & 4  \exp (- 2 x^{1/b}) \, E \left| \int_0^1 H_{u-}^{(x)} \,
dL_u^\sharp \right|^2 \nonumber \\
& = & 4  \exp (- 2 x^{1/b}) \,  \int_0^1 E |H_{u-}^{(x)}|^2 \,
dL_u^\sharp \; \mbox{\rm Var}(L_1^\sharp) \nonumber \\
& \leq & 4  \exp (- 2 x^{1/b}) \exp (x^{1/b})\; \mbox{\rm
Var}(L_1^\sharp)
\nonumber \\
& = & 4  \exp (-  x^{1/b}) \,(\sigma^2_L + \int_{|y|\leq 1} y^2\,
\nu_L(dy)), \nonumber
\end{eqnarray}
where we used  \cite[Ex. 25.12]{sato} to express the
variance $\mbox{\rm Var}(L_1^\sharp)$ in terms of the characteristic
triplet. Combining this with \eqref{eq-8.13}, we obtain
\begin{eqnarray*}
\lefteqn{ E \left( \log^+ \sup_{0<s\leq 1} \left| \int_0^s H_{u-} \,
dL_u^\sharp \right|\right)^b } \\
& \leq & 4 \left(\sigma^2_L + \int_{|y|\leq 1} y^2 \, \nu_L(dy)\right)
\int_0^\infty \exp (-x^{1/b}) \, dx + \int_0^\infty P \left( \left(
\log^+
\sup_{0\leq s \leq 1} |H_s|\right)^b > x 2^{-b} \right) \, dx \\
& = & 4 \left(\sigma^2_L + \int_{|y|\leq 1} y^2 \, \nu_L(dy)\right)
\int_0^\infty \exp (-x^{1/b}) \, dx + 2^b \, E \left( \log^+
\sup_{0\leq s \leq 1} |H_s| \right)^b,
\end{eqnarray*}
establishing \eqref{eq-8.12} for $L^\sharp$.

In order to obtain \eqref{eq-8.12} for $L^\flat$, denote
$$R_t := |\gamma_L| t + \sum_{0<s \leq t} |\Delta L_s^\flat|.$$
Then $R=(R_t)_{t\geq 0}$ is a subordinator and
\begin{eqnarray}
\lefteqn{ \left( \log^+ \sup_{0\leq s \leq 1} \left| \int_0^s H_{u-}
\, dL_u^\flat \right| \right)^b} \nonumber \\
& \leq & \left( \log^+ \left( R_1 \sup_{0\leq s \leq 1} |H_s|
\right) \right)^b \nonumber \\
& \leq & \left( \log^+ R_1 + \log^+ \sup_{0<s\leq 1} |H_s|\right)^b
\nonumber \\
& \leq & (2^{b-1} \vee 1) (\log^+ R_1)^b + (2^{b-1} \vee 1) \left(
\log^+ \sup_{0\leq s \leq 1} |H_s| \right)^b . \label{eq-8.14}
\end{eqnarray}
Since the function $x\mapsto (\log (x\vee e))^b$ is
submultiplicative (cf. Sato~\cite[Prop. 25.4]{sato}), it follows
from the proof of Theorem~25.3 in Sato~\cite{sato} that there is a
constant $C_2 = C_2(b)$, depending only on $b$, such that
\begin{equation*}
E \left( \log \left( e \vee \sum_{0<s\leq 1} |\Delta L_s^\flat|
\right)\right)^b \leq \exp \left\{ C_2 \int_{|x|>1} (\log^+ |x|)^b
\, \nu_L(dx) \right\}.
\end{equation*}
Hence, there is a constant $C_4 = C_4(b) \in (0,\infty)$ such that
\begin{eqnarray}
\lefteqn{ E (\log^+ R_1)^b } \nonumber \\
& \leq & 1 + E (\log (e \vee R_1))^b \nonumber  \\
& \leq & C_4 \left( 1 + (\log^+ |\gamma_L|)^b + \exp \left\{ C_2
\int_{|x|>1} (\log^+ |x|)^b \, \nu_L(dx) \right\} \right). \nonumber
\end{eqnarray}
Together with \eqref{eq-8.14} this gives \eqref{eq-8.12} for
$L^\flat$.
\end{proof}

\begin{proof}[Proof of Theorem \ref{thm-continuity}]
Recall that for any real numbers $a$ and $b$ and $\delta>0$ it holds
$$|a+b|^{1+\delta}\leq C_\delta (|a|^{1+\delta}+|b|^{1+\delta})$$
for some constant $C_\delta$. Using this together with Doob's martingale inequality (c.f. \cite[Eq. (25.16)]{sato}) and Jensen's inequality we obtain
\begin{eqnarray*}
 E[\sup_{0<s\leq 1} |\xi_s^{(n)}|^{1+\delta} ]
&\leq& C_\delta \left( E[\sup_{0<s\leq 1} |\xi_s^{(n)}-sE[\xi_1^{(n)}] |^{1+\delta}] +  |E[\xi_1^{(n)}]|^{1+\delta}\right)\\
&\leq& C_\delta \left( 8 E[|\xi_1^{(n)}-E[\xi_1^{(n)}] |^{1+\delta}] +  |E[\xi_1^{(n)}]|^{1+\delta}\right)\\
&\leq& 8 C_\delta^2 E[|\xi_1^{(n)}|^{1+\delta}] + (8C_\delta^2+C_\delta) |E[\xi_1^{(n)}]|^{1+\delta}\\
&\leq& (16C_\delta^2+C_\delta) E[|\xi_1^{(n)}|^{1+\delta}].
\end{eqnarray*}
Hence from \eqref{contcond1} we conclude
\begin{equation} \label{contcond5b}
 \sup_{n\in\bN} E [(\sup_{0<s\leq 1} |\xi_s^{(n)}|)^{1+\delta} ]<\infty
\end{equation}
and therefore also
\begin{equation} \label{contcond5}
 \sup_{n\in\bN} E [(\sup_{0<s\leq 1} |\xi_s^{(n)}\vee 0|)^{1+\delta} ]<\infty .
\end{equation}
Denote by $(\gamma_{\eta^{(n)}}, \sigma^2_{\eta^{(n)}},
\nu_{\eta^{(n)}})$ and $(\gamma_\eta,\sigma_\eta^2,\nu_\eta)$ the
characteristic triplets of $\eta^{(n)}$ and $\eta$, respectively.
Denote by $h$ the continuous truncation function $h(x) = x
\mathbf{1}_{|x|\leq 1} + (2-|x|) \mbox{sgn}(x) \, \mathbf{1}_{|x|\in
(1,2]}$. Set
$$\beta_{\eta^{(n)}} := \gamma_{\eta^{(n)}} + \int_{[-2,2]} (h(x) - x
\mathbf{1}_{|x| \leq 1}) \, \nu_{\eta^{(n)}} (dx) =
\gamma_{\eta^{(n)}} + \int_{[-2,2]} x\left(\frac{h(x)}{x} -
\mathbf{1}_{|x| \leq 1}\right) \, \nu_{\eta^{(n)}} (dx),$$ i.e. the
constant term in the L\'evy-Khintchine triplet of $\eta^{(n)}$ with
respect to the truncation function $h$ (c.f. \cite[Eqs. (8.5),
(8.6)]{sato}). Define $\beta_\eta$ similarly. Since $\eta^{(n)}_1
\stackrel{d}{\to} \eta_1$, it follows from \cite[Thm. VII.2.9, p.396]{JacodShiryaev} that $\beta_{\eta^{(n)}} \to
\beta_\eta$,
\begin{eqnarray*}
 \lefteqn{\sigma^2_{\eta^{(n)}} + \int_{|x|\leq 1} x^2 \nu_{\eta^{(n)}} (dx)
+ \int_{1< |x| \leq 2} (2-|x|)^2 \, \nu_{\eta^{(n)}} (dx) }\\
&\to& \sigma^2_{\eta} + \int_{|x|\leq 1} x^2 \nu_{\eta} ( dx) + \int_{1<
|x| \leq 2} (2-|x|)^2 \, \nu_{\eta} (dx)
\end{eqnarray*}
and $\int_{\bR} f(x) \,
\nu_{\xi^{(n)}} (dx) \to \int_{\bR} f(x) \, \nu_\xi (dx)$ as
$n\to\infty$ for every continuous bounded function $f$ vanishing in
a neighbourhood of zero. In particular,
\begin{equation} \label{eq-8.17}
\sup_{n\in \bN} \sigma^2_{\eta^{(n)}} < \infty,
\quad \sup_{n\in \bN} \int_{[-1,1]} x^2 \, \nu_{\eta^{(n)}} (dx) <
\infty \quad \mbox{and} \quad \sup_{n\in \bN} |\gamma_{\eta^{(n)}}|
< \infty.
\end{equation}
Applying Lemma~\ref{lem-8.5} with $b=1+\delta$ and $H_s = 1$ and
using \eqref{contcond6} then shows that $\sup_{n\in \bN} E (\log^+
|\eta_1^{(n)}|)^{1+\delta}<\infty$, and hence that $E (\log^+
|\eta_1|)^{1+\delta} < \infty$ by Fatou's lemma for weak convergence
(cf. Kallenberg~\cite[Lem.~4.11]{kallenberg}), and similarly we
obtain $E |\xi_1|^{1+\delta} < \infty$ from \eqref{contcond5b}.

Since $E \log^+|\eta_1| < \infty$, $E \xi_1 > 0$, $E \log^+
|\eta_1^{(n)}| < \infty$
and $E \xi_1^{(n)} > 0$, the integrals $\int_0^\infty e^{-\xi_{s-}}
\, d\eta_s$ and
 $\int_0^\infty e^{-\xi_{s-}^{(n)}} \, d\eta_s^{(n)}$ converge
almost surely absolutely (cf. \cite[Thm. 2]{ericksonmaller05}).
Writing
$$\int_0^\infty e^{-\xi_{s-}^{(n)}} \, d\eta_s^{(n)} = \int_0^1
e^{-\xi_{s-}^{(n)}} \, d\eta_s^{(n)} + e^{-\xi_1^{(n)}}
\int_1^\infty e^{-(\xi_{s-}^{(n)} - \xi_1^{(n)})} \, d(\eta_s^{(n)}
- \eta_1^{(n)}),$$ we have
$$\int_0^\infty e^{-\xi_{s-}^{(n)}} \, d\eta_s^{(n)} \stackrel{d}{=}
\sum_{k=0}^\infty \left( \prod_{i=0}^{k-1} A_i^{(n)} \right)
B_k^{(n)}$$ with some i.i.d. sequences $(A_k^{(n)}, B_k^{(n)})_{k\in
\bN_0}$ such that
$$(A_0^{(n)}, B_0^{(n)}) {=} (e^{-\xi_1^{(n)}}, \int_0^1
e^{-\xi_{s-}^{(n)}} \, d\eta_s^{(n)}),$$ and a similar statement
holds for $\int_0^\infty e^{-\xi_{s-}} \, d\eta_s$ with
$(A_k,B_k)_{k\in \bN_0}$ i.i.d. such that $(A_0,B_0) = (e^{-\xi_1},
\int_0^1 e^{-\xi_{s-}} \, d\eta_s)$.

Now, to apply Proposition \ref{prop-brandt}, we have to check its
conditions on the sequences $(A_0^{(n)})_{n\in\NN}$ and
$(B_0^{(n)})_{n\in\NN}$ which we shall do in the following.

Since $(\xi_1^{(n)}, \eta_1^{(n)}) \stackrel{d}{\to}
(\xi_1,\eta_1)$, $n\to\infty,$ it follows from
\cite[Cor. VII.3.6, p. 415]{JacodShiryaev} that $(\xi^{(n)},\eta^{(n)})
\stackrel{\cL}{\to} (\xi,\eta)$, where ``$\stackrel{\cL}{\to}$''
denotes convergence in the Skorokhod topology. Additionally, the
sequences $(\xi^{(n)}), n\in\bN$, and $(\eta^{(n)}), n\in \bN$,
satisfy the P-UT condition (cf.
\cite[Def.~VI.6.1, p.~377]{JacodShiryaev}). To see this, let $h:\bR^2 \to \bR^2$
be a continuous bounded function satisfying $h(x) = x$ in a
neighbourhood of 0. Then if $(\gamma_h,A,\nu)_h$ is the
L\'evy-Khintchine triplet of $(\xi,\eta)$ with respect to $h$ (we
use the notations here as in \cite[Eq. II.4.21,
p.~107]{JacodShiryaev}), then $(\gamma_h t, At, dt \, \nu(dx))$,
$t\geq 0$, is the semimartingale characteristic of $(\xi,\eta)$ with
respect to $h$, cf.  \cite[Cor.~II.4.19, p. 107]{JacodShiryaev}. A
similar statement holds for $(\xi^{(n)},\eta^{(n)})$. Since
$(\xi^{(n)}, \eta^{(n)}) \stackrel{\cL}{\to} (\xi,\eta)$ as
$n\to\infty$, the sequence $(\xi^{(n)}, \eta^{(n)})$, $n\in\bN$, is
tight. Furthermore, since again by \cite[Cor. VII.3.6, p.
415]{JacodShiryaev}, $\gamma^{(n)}_h \to \gamma_h$ as $n\to\infty$,
and since the total variation of $s\mapsto \gamma^{(n)}_h s$ on
$[0,t]$ is $|\gamma^{(n)}_h| t$, condition (iii) of \cite[Thm. VI.6.15, p. 380]{JacodShiryaev} is satisfied, and it follows from
 \cite[Thm. VI.6.21, p. 382]{JacodShiryaev} that
$(\xi^{(n)},\eta^{(n)})$, $n\in \bN$, is P-UT. Then also
$(\eta^{(n)})_{n\in \bN}$ is P-UT (cf. \cite[Eq. VI.6.3, p.
377]{JacodShiryaev}).\\
From \cite[Thm.~VI.6.22, p. 383]{JacodShiryaev} it now follows
that
\begin{equation} \label{eq-8.8}
(\xi^{(n)}, \eta^{(n)}, \int_0^{\cdot} e^{-\xi_{s-}^{(n)}} \,
d\eta_s^{(n)}) \stackrel{\cL}{\to} (\xi,\eta,\int_0^{\cdot}
e^{-\xi_{s-}} \, d\eta_s), \quad n\to\infty,
\end{equation}
in the Skorokhod topology. Since none of the components has a
discontinuity at fixed $t\geq 0$ with positive probability, this
implies
\begin{equation}\label{brandtcond1}
 (A_0^{(n)}, B_0^{(n)})\overset{d}\to (A_0,B_0), \quad n\to \infty.
\end{equation}

By assumption we have $\log|A_0^{(n)}| = -\xi_1^{(n)}
\stackrel{d}\to -\xi_1 = \log|A_0|$. Since additionally the sequence
$(\log|A_0^{(n)}|)_{n\in\NN}$ is uniformly integrable by
\eqref{contcond1} (see e.g. \cite[Condition (3.18)]{billingsley}),
this yields by \cite[Thm. 3.5]{billingsley}
\begin{equation}\label{brandtcond2}
 E[\log |A_0^{(n)}|] \to E[\log|A_0|], \quad n\to \infty.
\end{equation}
Since \eqref{contcond5} implies $\sup_n E[|\log^+
|A_0^{(n)}||^{1+\delta}]<\infty$ we obtain similarly
\begin{equation}\label{brandtcond3}
 E[\log^+ |A_0^{(n)}|] \to E[\log^+|A_0|], \quad n\to \infty.
\end{equation}
Also, it is obvious that \eqref{contcond3} and \eqref{contcond1}
yield
\begin{equation}\label{brandtcond4}
-\infty<E[\log |A_0|]<0 \quad \mbox{and} \quad -\infty<E[\log
|A_0^{(n)}|]<0.
\end{equation}
Finally, observe that \eqref{contcond5} implies
\begin{equation*} \label{eq-cont1}
 \sup_{n\in \bN}  E[\log^+ \sup_{0<s\leq 1} |e^{-\xi_s^{(n)}}|]^{1+\delta}<\infty
\end{equation*}
which, together with \eqref{contcond6} and \eqref{eq-8.17}, yields
by Lemma \ref{lem-8.5}
$$\sup_{n\in \bN} E [\log^+ |B_0^{(n)}|]^{1+\delta} < \infty.$$
Again, this gives $E [\log^+ |B_0|]^{1+\delta} < \infty$ and
\begin{equation}\label{brandtcond5}
 E[\log^+ |B_0^{(n)}|] \to E[\log^+|B_0|]<\infty, \quad n\to \infty.
\end{equation}
Now, by Proposition \ref{prop-brandt} we obtain the stated result
from \eqref{brandtcond1}, \eqref{brandtcond2}, \eqref{brandtcond3},
\eqref{brandtcond4} and \eqref{brandtcond5}.
\end{proof}

From the above theorem we immediately obtain the following corollary
on injectivity of $\Phi_\xi$ and $\tilde{\Phi}_\eta$. Observe that
the conditions in part $(i)$ have been violated in Example
\ref{ex-2}.

\begin{corollary}\begin{enumerate}
\item Let $(\xi_t)_{t\geq 0}$ be a L\'evy process such that $E[\xi_1]>0$ and $E[|\xi_1|^{1+\delta}]<\infty$ for some $\delta>0$.
Let $(\eta^{(n)})_{n\in\NN}$ be a sequence of L\'evy processes such
that $\eta_1^{(n)}\overset{d}\to \eta_1$ as $n\to\infty$,
$\cL(\eta^{(n)}_1)\in D_\xi$, $\cL(\eta_1)\in D_\xi$ and
$$\sup_{n\in\NN} \int_{|x|>1} (\log^+|x|)^{1+\delta} \nu_{\eta^{(n)}}(dx)<\infty.$$
Then $\Phi_\xi(\cL(\eta^{(n)}_1)) \overset{w}\to \Phi_\xi(\cL(\eta_1))$ as
$n\to\infty$.
\item Let $(\eta_t)_{t\geq 0}$ be a L\'evy process such that $E[\log^+|\eta_1|^{1+\delta}]<\infty$ for some $\delta>0$.
Let $(\xi^{(n)})_{n\in\NN}$ be a sequence of L\'evy processes such
that $\xi_1^{(n)}\overset{d}\to \xi_1$ as $n\to\infty$,
$\cL(\xi^{(n)}_1)\in \tilde{D}_\eta$, $\cL(\xi_1)\in \tilde{D}_\eta$, $E[\xi_1^{(n)}]>0$, $E[\xi_1]>0$
and
$$\sup_{n\in\NN} E[|\xi_1^{(n)}|^{1+\delta}]<\infty.$$
Then $\tilde{\Phi}_\eta(\cL(\xi^{(n)}_1)) \overset{w}\to
\tilde{\Phi}_\eta(\cL(\xi_1))$ as $n\to\infty$.
\end{enumerate}
\end{corollary}

\section*{Acknowledgement}
The authors would like to thank Makoto Maejima for fruitful
discussions which initiated the investigations of this paper.

\end{document}